\documentclass[leqno,11pt,twoside]{amsart}
\usepackage{geometry}
\usepackage{physics}
\usepackage{times}

\usepackage[all]{xy} 
\usepackage{esvect}
\usepackage{amsmath, amssymb, amsfonts, latexsym, mdwlist, amsthm,amscd}
\usepackage{caption, subcaption}
\usepackage{tikz}
\usetikzlibrary{calc,trees,positioning,arrows,chains,shapes.geometric,%
	decorations.pathreplacing,decorations.pathmorphing,shapes,%
	matrix,shapes.symbols}

\tikzset{
	>=stealth',
	punktchain/.style={
		rectangle,
		rounded corners,
		draw=black, thick,
		minimum height=3em,
		text centered,
		on chain},
	line/.style={draw, thick, <-},
	element/.style={
		tape,
		top color=white,
		bottom color=blue!50!black!60!,
		minimum width=8em,
		draw=blue!40!black!90, very thick,
		text width=10em,
		minimum height=3.5em,
		text centered,
		on chain},
	every join/.style={->, thick,shorten >=1pt},
	decoration={brace},
	tuborg/.style={decorate},
	tubnode/.style={midway, right=2pt},
}

\usepackage{url}

\usepackage{verbatim}
\usepackage{mathtools}

\usepackage{pgf,tikz}
\usepackage{tikz-cd}
\usetikzlibrary{calc,matrix,patterns,shadings,backgrounds,fit}
\pgfdeclarelayer{myback}
\pgfsetlayers{myback,background,main}
\usepackage{stmaryrd}
\usepackage{diagbox}

\usepackage[bookmarks, colorlinks, breaklinks, pdftitle={Clifford index of curves on K3 surfaces},
pdfauthor={Soheyla Feyzbakhsh, Chunyi Li}]{hyperref}
\hypersetup{linkcolor=blue,citecolor=blue,filecolor=black,urlcolor=blue}

\usepackage{paralist}
\setdefaultenum{(a)}{(i)}{}{}
\usepackage{enumitem} 

\def\C{\ensuremath{\mathbb{C}}}

\def\R{\ensuremath{\mathbb{R}}}
\def\Z{\ensuremath{\mathbb{Z}}}

\def\alg{\mathrm{alg}}

\def\ch{\mathop{\mathrm{ch}}\nolimits}

\def\Coh{\mathop{\mathrm{Coh}}\nolimits}

\def\deg{\mathop{\mathrm{deg}}}

\def\dim{\mathop{\mathrm{dim}}\nolimits}

\def\ext{\mathop{\mathrm{ext}}\nolimits}
\def\Ext{\mathop{\mathrm{Ext}}\nolimits}

\def\HH{\mathrm{H}}

\def\Hom{\mathop{\mathrm{Hom}}\nolimits}

\def\Ker{\mathop{\mathrm{Ker}}\nolimits}

\def\min{\mathop{\mathrm{min}}\nolimits}

\def\Pic{\mathop{\mathrm{Pic}}}

\def\rk{\mathop{\mathrm{rk}}}

\def\Stab{\mathop{\mathrm{Stab}}\nolimits}

\def\Db{\mathrm{D}^{b}}

\def\star{(*)}

\newcommand\TFILTB[3]{%
\xymatrix@=1pc{
{0 = {#1}_0} \ar[rr]&&
{{#1}_1} \ar[rr]\ar[ld] &&
{{#1}_2} \ar[r]\ar[ld] &
{\cdots} \ar[r] & { {#1}_{#3-1}} \ar[rr] &&
{{#1}_{#3} = {#1}} \ar[ld]
\\
& *{{#2}_1} \ar@{.>}[ul] &&
{{#2}_2} \ar@{.>}[ul] & &&&
{{#2}_{{#3}}} \ar@{.>}[ul]
}}

\def\abs#1{\left\lvert#1\right\rvert}

\newcommand\stv[2]{\left\{#1\,\colon\,#2\right\}}

\makeatletter
\newtheorem*{rep@theorem}{\rep@title}
\newcommand{\newreptheorem}[2]{%
\newenvironment{rep#1}[1]{%
 \def\rep@title{#2 \ref{##1}}%
 \begin{rep@theorem}}%
 {\end{rep@theorem}}}
\makeatother

\newtheorem{Thm}{Theorem}[section]
\newreptheorem{Thm}{Theorem}
\newtheorem{Prop}[Thm]{Proposition}

\newtheorem{Lem}[Thm]{Lemma}

\newtheorem{Cor}[Thm]{Corollary}
\newreptheorem{Cor}{Corollary}

\newreptheorem{Con}{Conjecture}

\newtheorem{thm-int}{Theorem}

\theoremstyle{definition}
\newtheorem{Def-s}[Thm]{Definition}
\newtheorem{Def}[Thm]{Definition}
\newtheorem{Rem}[Thm]{Remark}

\def\C{\ensuremath{\mathbb{C}}}

\def\R{\ensuremath{\mathbb{R}}}
\def\Z{\ensuremath{\mathbb{Z}}}

\def\cA{\ensuremath{\mathcal A}}

\def\cF{\ensuremath{\mathcal F}}

\def\cO{\ensuremath{\mathcal O}}

\def\cT{\ensuremath{\mathcal T}}

\def\cW{\ensuremath{\mathcal W}}

\def\vv{\ensuremath{\mathbf v}}

\def\mH{\ensuremath{\mathrm{H}}}

\def\UUU{\mathfrak U}

\usepackage{todonotes}

\def\Halg{H^*_{\alg}(X, \Z)}
\def\HalgR{H^*_{\alg}(X, \R)}

\def\sab{\sigma_{\alpha, \beta}}

\begin{document}

\title[Higher rank Clifford indices of curves on a K3 surface]{Higher rank Clifford indices of curves on a K3 surface}

\author{Soheyla Feyzbakhsh}
\address{S. F.: \\
School of Mathematics and Maxwell Institute,
University of Edinburgh,
James Clerk Maxwell Building,
Peter Guthrie Tait Road, Edinburgh, EH9 3FD,
United Kingdom}
\email{s.feyzbakhsh@ed.ac.uk}
\urladdr{http://www.maths.ed.ac.uk/~s1373772/}

\author{Chunyi Li}
\address{C. L.:\\
Mathematics Institute, University of Warwick,
Coventry, CV4 7AL,
United Kingdom}
\email{C.Li.25@warwick.ac.uk}
\urladdr{https://sites.google.com/site/chunyili0401/}

\keywords{Clifford Index, Bridgeland stability conditions, K3 surface}

\begin{abstract} 
Let $(X,H)$ be a polarized K3 surface with $\mathrm{Pic}(X) = \Z H$, and let $C\in |H|$ be a smooth curve of genus $g$. We give an upper bound on the dimension of global sections of a semistable vector bundle on $C$. This allows us to compute the higher rank Clifford indices of $C$ with high genus. In particular, when $g\geq r^2\geq 4$, the rank $r$ Clifford index of $C$ can be computed by the restriction of Lazarsfeld-Mukai bundles on $X$ corresponding to line bundles on the curve $C$. This is a generalization of the result by Green and Lazarsfeld for curves on K3 surfaces to higher rank vector bundles. 
We also apply the same method to the projective plane and show that the rank $r$ Clifford index of a degree $d(\geq 5)$ smooth plane curve is $d-4$, which is the same as the Clifford index of the curve.

\end{abstract}

\date{\today}

\maketitle

\setcounter{tocdepth}{1}
\tableofcontents
\def\mhk{M^h_k} 
\def\lrd{V^r_d(\abs{H})}
\def\ext{\mathrm{ext}}
\def\Ext{\mathrm{Ext}}
\def\clf{\mathrm{Cliff}}
\def\rg{\frac{r}{g}}
\def\gr{\frac{g}{r}}
\def\fgr{\left\lfloor\frac{g}{r}\right\rfloor}
\def\lf{\left\lfloor}
\def\rf{\right\rfloor}
\def\pp{\mathbf{P}^2}
\def\vvv{\tilde{\mathbf{v}}}

\section{Introduction}

Let $\UUU_C(r,d)$ be the set of semistable rank $r$-vector bundles of degree $d$ on a smooth curve $C$. For $E\in \UUU_C(r,d)$, its \textit{Clifford index} is defined as 
\[\clf(E)=\frac{d}{r}-\frac{2}{r}h^0(C,E)+2.\]
By the higher rank Clifford Theorem (\cite[Theorem 2.1]{Brambila-Paz1995Geography}), when $0\leq d\leq r(g-1)$, the index $\clf(E)$ is non-negative. The \textit{rank $r$ Clifford index} of $C$, in this paper, is defined as:
\[\clf_r(C)\coloneqq \min \{ \clf(E)| E\in \UUU_C(r,d),\, d\leq r(g-1), h^0(C,E)\geq 2r \}.\]
The main result is as follows. 
\begin{Thm}
Let $(X,H)$ be a smooth polarized K3 surface with $\mathrm{Pic}(X) = \Z H$, and let $C$ be a smooth curve of genus $g$ in the linear system $|H|$. Let $E$ be a slope semistable rank $r$-vector bundle of degree $d$ on the curve $C$ such that $d\leq r(g-1)$. Then we have the bound for the dimension of the global sections of $E$:
\begin{align}
h^0(C,E) < r+\frac{g}{4r(g-1)^2} d^2 +\rg.\label{eq:mainthm}
\end{align}
When $r\geq 2$ and $g\geq r^2$, the rank $r$ Clifford index of $C$ 
\[\clf_r(C) = \frac{2}{r}(g-1)-\frac{2}{r}\fgr.\]
\label{main1}
\end{Thm}
The upper bound for $h^0(C,E)$ in Theorem \ref{main1} is much stronger than the higher rank Clifford Theorem, which says $h^0(C,E)\leq r +\frac{d}{2}$. The bound is not far from the sharp bound, see Remark \ref{rem:sharp}. For a smooth curve $C$ of genus g, several upper bounds for the dimension of global sections of vector bundles of low slope $\mu= d/r$ have been introduced in \cite{Brambila-Paz1995Geography,Mercat-brill-noether-1,Mercat-stable-bundles-of-slope-2}, which are also included in \cite{Lange2015g4}.  Sharp bounds for the case $g \leq 6$ and $\mu<2$ have been determined in \cite{Brambila-Paz1995Geography,Mercat-brill-noether-1,Mercat-stable-bundles-of-slope-2,Lange2015g4,Lange2015g5,Lange2017g6}. The upper bound \eqref{eq:mainthm} is in general stronger than the bounds in these previous papers unless $g\leq 6$ or $\mu\leq 2$. 

For $r=2$, the second statement of Theorem \ref{main1} gives
\begin{align*}
\clf_2(C)=\clf(C) =\lf\frac{g-1}2 \rf,
\end{align*}
so we re-obtain the result \cite[Theorem 1.3]{Bakker2015Mercat}. Also for $r \geq 3$, we have 
\begin{align*}
\clf_r(C) < \clf(C) =\lf\frac{g-1}2 \rf. 
\end{align*}
This indicates the failure of the Mercat's conjecture in  \cite{Mercat2002CLIFFORD} for $C$ which states the higher ranks Clifford indices of the curve $C$ are equal to $\clf(C)$. 

Let $A$ be a globally generated line bundle on the curve $C\subset X$, the Lazarsfeld-Mukai bundle $E_{C,A}$ on $X$ is defined via the exact sequence
\begin{equation*}
0 \rightarrow E_{C,A}^{\vee} \rightarrow H^0(C,A) \otimes \mathcal{O}_X \xrightarrow{\textbf{ev}} A \rightarrow 0. 
\end{equation*}
In all cases in the second part of Theorem \ref{main1}, there exists a line bundle $A$ on the curve $C$ such that the rank $r$-Clifford index is computed by the restriction of the corresponding Lazarsfeld-Mukai bundle on the K3 surface $X$. We expect this result holds without the assumption on the Picard group of $X$. This can be viewed as a generalization for the result of Green and Lazarsfeld in \cite{Green-Lazarsfeld-clifford} which says that for a curve $C$ on a smooth K3 surface with $\clf(C) < \lf\frac{g-1}2 \rf$, the Clifford index can be computed by the restriction of a line bundle on the K3 surface. \par
Our argument can be generalized to curves on other surfaces, especially when the surface admits a stronger Bogomolov-Gieseker inequality.  Examples of such surfaces include the projective plane, del Pezzo surfaces and quintic surfaces. We explain more details for smooth plane curves in Section \ref{section:p2}. In particular, we show that the first part of the Mercat's conjecture \cite{Mercat2002CLIFFORD} holds for smooth plane curves:
\begin{Thm}[Corollary \ref{cor:mercatforp2}]\label{thmintro:mercatforp2}
Let $C$ be a degree $l(\geq 5)$ smooth irreducible plane curve, then 
\[\clf_r(C)=l-4,\]
for any positive integer $r$.
\end{Thm}
Another concrete example for curves on degree four del Pezzo surfaces is computed in \cite{Li:BGquintic}. The Clifford type inequality for such curves is the key ingredients in proving the existence of Bridgeland stability conditions on smooth quintic threefolds. 
\subsection{Approach}
The main tool in this paper is the notion of \textit{stability condition} introduced by Bridgeland \cite{Bridgeland:Stab}. In general, such a stability condition $\sigma=(\cA, Z)$ is defined on a $\C$-linear triangulated category $T$, and is consisting of a heart structure $\cA$ and a central charge $Z \colon K(T)\rightarrow \C$, which is a group homomorphism from the Grothendieck group to complex numbers. The space of stability conditions on $T$ forms a complex manifold which admits a wall and chamber decomposition for any fixed object $E \in T$. In this paper, the triangulated category $T$ will always be the bounded derived category $D^b(X)$ of coherent sheaves on a surface $X$. We will only make use of a real two-dimensional subspace of stability conditions on $D^b(X)$. \par

Let $\iota \colon  C\hookrightarrow X$ be the embedding of a smooth curve $C$ into the surface $X$, and let $E$ be a semistable vector bundle on the curve $C$. In \cite{Fe17}, a new upper bound for the dimension of global sections of objects in $D^b(X)$ has been introduced. This states the dimension of global sections of $\iota_*E$ can be bounded by the length of the Harder-Narasimhan polygon at a limit point $\sigma_0$ where $Z(\mathcal{O}_X) \rightarrow 0$. The Harder-Narasimhan polygon geometrically represents the slopes and degrees of the Harder-Narasimhan factors of $\iota_*E$ with respect to $\sigma_0$.  One of the key parts of the paper is to describe the position of the wall for $\iota_*E$ that bounds the large volume limit chamber at where $\iota_*E$ is stable.  Describing the wall that bounds the large volume limit will enable us to control the length of this Harder-Narasimhan polygon at $\sigma_0$ effectively and get the bound for the dimension of global sections of the vector bundle $E$.

\subsection*{Generalization} All our results hold slightly more general for a polarized K3 surface $(X, H)$ satisfying the following:
\begin{description}
\item[Assumption (*)] $H^2$ divides $H.D$ for all curve classes $D$ on $X$.
\end{description}
To simplify the presentation, we explain our entire argument in the case of Picard rank one and  then explain in Section \ref{sec:6.1} how to extend the arguments to this situation.

\subsection*{Acknowledgments} The authors would like to thank Arend Bayer for many helpful conversations. Both authors were supported by the ERC starting grant `WallXBirGeom' 337039. Chunyi Li is a Leverhulme Early Career Fellow and would like to acknowledge the Leverhulme Trust for the support.

\def\nab{\nu_{\beta, \alpha}}
\def\sab{\sigma_{\beta, \alpha}}
\def\zab{Z_{\beta, \alpha}}

\section{Review: stability conditions, wall-crossings}
\label{sect:background}
Let $(X,H)$ be a smooth polarized K3 surface over $\mathbb{C}$ with Pic$(X) = \Z H$. In this section, we review the description of a slice of the \textit{space  of stability conditions} $\Stab(X)$ on $D^b(X)$ given in \cite[Section 1-7]{Bridgeland:K3}.\par

Given an object $E \in D^b(X)$, we write $\ch(E) = (\rk(E), \ch_1(E), \ch_2(E)) \in H^*(X, \Z)$ for its Chern characters. We write $H^*_{\mathrm{alg}}(X, \Z)$ for its algebraic part, in other words, the image of $\ch(-)$. 
The slope of a coherent sheaf $E \in \Coh X$ is defined by 
\[ \mu_{H}(E) := \begin{cases}
\frac{H.\ch_1(E)}{H^2\rk(E)} & \text{if $\rk(E) > 0$} \\
+\infty & \text{if $\rk(E) = 0$}.
\end{cases}\]
This leads to the usual notion of $\mu_H$-stability. For any $\beta \in \mathbb{R}$, we have the following torsion pair in $\Coh X$
\begin{align*}
\cT^{\beta} & :=
\langle E \colon \text{$E$ is $\mu_{H}$-semistable with $\mu_{H}(E) > \beta$ } \rangle, \\
\cF^{\beta} & :=
\langle E \colon \text{$E$ is $\mu_{H}$-semistable with $\mu_{H}(E) \le \beta$} \rangle,
\end{align*}
where $\langle - \rangle$ denotes the extension-closure. Following \cite{Happel-al:tilting,Bridgeland:K3}, this lets us define a new heart of a bounded t-structure in $\Db(X)$ as follows:
\[
\Coh^\beta X := \langle \cF^\beta[1], \cT^\beta \rangle =
\stv{E}{\mH^{-1}(E) \in \cF^\beta, \mH^0(E) \in \cT^\beta, \mH^i(E) = 0 \text{ for $i \neq 0, -1$}}.
\]
For any pair $(\beta,\alpha) \in \R^2$, we define the central charge $\zab \colon K(X) \to \C$ by
\begin{equation} \label{eq:Zab}
\zab(E) := -\ch_2(E) + \alpha \rk(E)  + i \bigg(\frac{H\ch_1(E)}{H^2} -\beta \rk(E) \bigg).
\end{equation}
Note that the function $\zab$, up to the action of GL$^{+}(2;\mathbb{R})$, is the same as the stability function defined in \cite[section 6]{Bridgeland:K3}. The function $\zab$ factors via the Chern character
\begin{equation} \label{eq:ch}
\ch \colon K(X) \to \Halg \cong \Z^3 \;\;\;, \;\;\;\ch(E) = \big(\rk(E), \ch_1(E), \ch_2(E)\big).
\end{equation}
The kernel of $\zab$ in $H^*_{\mathrm{alg}}(X, \R)$ under the basis  $\{\rk,\ch_1,\ch_2\}$ is spanned by $(1,\beta H,\alpha)$. 
\begin{Def} \label{def:gcurve}
Let $\gamma:\R\rightarrow \R$ be a $1$-periodic function such that for $x\in [-\frac{1}{2},\frac{1}{2}]$ is defined as 
\[ \gamma(x) := \begin{cases}
(1-x^2) & \text{if $x\neq 0$} \\
0 & \text{if $x = 0$}.
\end{cases}\]
Let $\Gamma(x):=\frac{H^2}{2}x^2-\gamma(x)$. By abuse of notations, we also denote the graph of $\Gamma$ by \textit{curve $\Gamma$} (see Figure \ref{fig:gammacurve}). 
	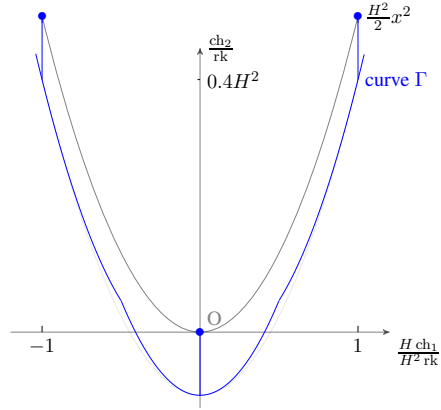
\begin{figure}[h]
\begin{center}
\scalebox{0.7}{
\tikzset{%
    add/.style args={#1 and #2}{
        to path={%
 ($(\tikztostart)!-#1!(\tikztotarget)$)--($(\tikztotarget)!-#2!(\tikztostart)$)%
  \tikztonodes},add/.default={.2 and .2}}
}
\begin{tikzpicture}[scale = 12]
\newcommand\XA{0.1}
\newcommand\obe{-0.3}
\newcommand\xa{-0.9}
\newcommand\xb{0.6}

\draw[->] [opacity=\XA*5] (-0.3,0) -- (0,0) node[above right] {}-- (0.3,0) node[below right, opacity =1] {$\frac{H\ch_1}{H^2\rk}$};
\draw (-0.25,0) to (-0.25,0.01);
\node [below] at (-0.25,0) {$-1$};
\draw (0.25,0) to (0.25,0.01);
\node [below] at (0.25,0) {$1$};
\draw[->][opacity=\XA*5] (0,-.12)-- (0,0) node [above right] {O} --  (0,0.45) node[right, opacity=1] {$\frac{\ch_2}{\rk}$};
\draw (0,0.4) to (-0.004,0.4);
\node [right] at (0,0.4) {$0.4H^2$};
\draw [opacity=\XA*5](-0.25,0.5) parabola bend (0,0) (0.25,0.5);
\draw [opacity=\XA](-0.25,0.4) parabola bend (0,-0.1) (0.25,0.4);

\draw[domain=-0.125:0.125,smooth,variable=\t,blue] plot ({\t},{8*\t*\t-0.1+1.6*\t*\t});
\draw[domain=-0.26:-0.125,smooth,variable=\t,blue] plot ({\t},{8*\t*\t+0.8*\t+1.6*\t*\t});
\draw[domain=0.125:0.26,smooth,variable=\t,blue] plot ({\t},{8*\t*\t-0.8*\t+1.6*\t*\t});
\draw[blue] (0,-0.1) to (0,0);
\draw[blue] (-0.25,0.49) to (-0.25,0.5);
\draw[blue] (0.25,0.49) to (0.25,0.5);
 \node[right,blue]  at (0.25,0.4) {curve $\Gamma$};
\draw[blue] (0.25,0.4)--(0.25,0.5);
\draw[blue] (-0.25,0.4)--(-0.25,0.5);
 \node[right]  at (0.25,0.5) {$\frac{H^2}{2}x^2$};
\node[blue] at(0,0) {$\bullet$};
\node[blue] at(0.25,0.5) {$\bullet$};
\node[blue] at(-0.25,0.5) {$\bullet$};
\end{tikzpicture} 
}
\end{center}
\caption{The Gamma curve.} \label{fig:gammacurve}
\end{figure}
\end{Def}
We first state Bridgeland's result describing stability conditions on $\Db(X)$, and then expand upon the statements.
\begin{Thm}[{\cite[Section 1]{Bridgeland:K3}}] \label{thm:stabconstr}
For any pair $ (\beta,\alpha) \in \R^2$ such that $\alpha>\Gamma(\beta)$, the pair $\sab := \left(\Coh^\beta X, \zab\right)$
defines a stability condition on $D^b(X)$. Moreover, the map from $\Gamma_+:= \{(\beta,\alpha)\in \R\times \R|\alpha>\Gamma (\beta)\} \to \Stab(X)$ is continuous.
\end{Thm}
\begin{Rem}
By \cite[Lemma 6.2, Proposition 7.1]{Bridgeland:K3}, we only need to check $\zab(F)\notin \R_{\leq 0}$ for all spherical sheaves $F\in \Coh(X)$. Assume for a spherical sheaf $F$ with slope in $[-\frac{1}2,\frac{1}2]$ and some $(\beta,\alpha)$, we have $\zab(F) \in \mathbb{R}$, then by definition 
$$\beta = \dfrac{H\ch_1(F)}{H^2\rk(E)}.$$
On the other hand, the spherical sheaf $F$ satisfies 
$$
2=\chi(F,F)= 2 \rk(F)\ch_2(F)-\ch_1(F)\ch_1(F)+2(\rk(F))^2 $$
Thus by Hodge index Theorem, we have
\begin{equation}\label{in.hodge}
\dfrac{H^2}{2}\left(\dfrac{H\ch_1(F)}{H^2\rk(F)}\right)^2 \geq \dfrac{\ch_1(F)^2}{2\rk(F)^2} = \dfrac{\ch_2(F)}{\rk(F)} +1 -\dfrac{1}{\rk(F)^2}.
\end{equation}
Also by assumption we have
$$\alpha > \Gamma\left(\beta = \dfrac{H\ch_1(F)}{H^2\rk(F)}\right) \geq \dfrac{H^2}{2}\left(\dfrac{H\ch_1(F)}{H^2\rk(F)}\right)^2 -1 +\dfrac{1}{\rk(F)^2}\geq \frac{\ch_2(F)}{\rk(F)}. $$
By inequality \eqref{in.hodge}, we have $\zab(F) \in \mathbb{R}_{>0}$.  
\end{Rem}

We first explain the notion of $\sab$-stability and the associated Harder-Narasimhan filtration. Consider the slope function
\[ \nab \colon \Coh^\beta X \to \R \cup \{+\infty\}, \quad \nab (E) := \begin{cases}
-\frac{\Re \zab(E)}{\Im \zab(E)} & \text{if $\Im \zab(E) > 0$} \\
+\infty & \text{if $\Im \zab(E) = 0$.}
\end{cases}
\]
This defines a notion of stability in $\Coh^\beta X$: an object $E\in \Coh^\beta X$ is $\sab$-(semi)stable if and and only if it is (semi)stable with respect to the the slope function $\nab$. Every object $E \in \Coh^\beta X$ admits a \textit{Harder-Narasimhan filtration} which is a finite sequence of objects in $\Coh^\beta X$,
\[0=F_0\subset F_1\subset F_2\subset\dots\subset F_k=E\]
whose factors $E_i:=F_i/F_{i-1}$ are  $\sab$-semistable and $\nab(E_1)>\nab(E_2)>\dots>\nab(E_k)$. We denote $\nab^+(E):= \nab(E_1)$ and $\nab^-(E):= \nab(E_k)$. The second part of Theorem \ref{thm:stabconstr} implies that the two-dimensional family of stability conditions $\sab$ satisfies wall-crossing as $\alpha$ and  $\beta$ vary. Consider the projection 
\begin{equation*}
pr \colon H^*_{\text{alg}}(X,\mathbb{Z}) \setminus \{\rk=0\} \rightarrow \mathbb{R}^2 \;\;\;,\;\;\; pr(\ch(E))=\left(\frac{H\ch_1(E)}{H^2\rk(E)},\frac{\ch_2(E)}{\rk(E)}\right).
\end{equation*} 
By abuse of notations, we use the same plane for the image of the projection $pr$ and the $(\beta, \alpha)$-plane. Note that the point $(\beta, \alpha)$ is equal to the projection $pr(\ker \zab)$ of the kernel of the central charge $\zab$ in $H^*_{\text{alg}}(X,\mathbb{Z})$. We will also write $pr(E)$ instead of $pr(\ch(E))$.

\begin{Rem}\label{rem:stob} 
\begin{enumerate}
\item For a stable object $E$ with respect to any stability condition $\sab$, by \cite[Theorem 2.15]{BM:walls},  the point $pr(E)$ is not in $\Gamma_{+} = \{(x,y) \in \mathbb{R}^2 \colon y>\Gamma(x)\}$.
\item The slope $\nab(E)$ is just the slope of the line crossing points $(\beta,\alpha)$ and $pr(E)$.
\end{enumerate}
\end{Rem}
\begin{figure}[h]
\begin{center}
\tikzset{%
    add/.style args={#1 and #2}{
        to path={%
 ($(\tikztostart)!-#1!(\tikztotarget)$)--($(\tikztotarget)!-#2!(\tikztostart)$)%
  \tikztonodes},add/.default={.2 and .2}}
}
\scalebox{0.6}{
\begin{tikzpicture}
     \newcommand\ad{2}
     \newcommand\yd{13}

       \coordinate (O) at (0,0.2);
	\draw (15,6.7)  node [below]{ $pr(\HalgR)$};
    
    \draw[domain=-0.5:0.5,smooth,variable=\t] plot ({\t+\yd},{0.6*\t*\t-0.1-\ad});
\draw[domain=-1.5:-0.5,smooth,variable=\t] plot ({\t+\yd},{0.6*\t*\t+0.2*\t-\ad});
\draw[domain=0.5:1.5,smooth,variable=\t] plot ({\t+\yd},{0.6*\t*\t-0.2*\t-\ad});
\draw[domain=-2.5:-1.5,smooth,variable=\t] plot ({\t+\yd},{0.6*\t*\t+0.3+0.4*\t-\ad});
\draw[domain=1.5:2.5,smooth,variable=\t] plot ({\t+\yd},{0.6*\t*\t+0.3-0.4*\t-\ad});
\draw[domain=-3.5:-2.5,smooth,variable=\t] plot ({\t+\yd},{0.6*\t*\t+0.8+0.6*\t-\ad});
\draw[domain=2.5:3.5,smooth,variable=\t] plot ({\t+\yd},{0.6*\t*\t+0.8-0.6*\t-\ad});
\draw[domain=-4:-3.5,smooth,variable=\t] plot ({\t+\yd},{0.6*\t*\t+1.5+0.8*\t-\ad});
\draw[domain=3.5:4,smooth,variable=\t] plot ({\t+\yd},{0.6*\t*\t+1.5-0.8*\t-\ad});
\draw[blue] (\yd,0-\ad) to (\yd,-0.1-\ad);
\draw[blue] (\yd+1,0.5-\ad) to (\yd+1,0.4-\ad);
\draw[blue] (\yd-1,0.5-\ad) to (\yd-1,0.4-\ad);
\draw[blue] (\yd+2,2-\ad) to (\yd+2,1.9-\ad);
\draw[blue] (\yd-2,2-\ad) to (\yd-2,1.9-\ad);
\draw[blue] (\yd+3,4.5-\ad) to (\yd+3,4.4-\ad);
\draw[blue] (\yd-3,4.5-\ad) to (\yd-3,4.4-\ad);
    \coordinate (V) at (16,-3);
	\draw (V) node {$\bullet$}  node[above right]{$pr(F)$};
    \coordinate (A) at (8.5,5.3);
	\draw (A)   node[above left]{$\mathcal W_{F}^i$};
    
    \coordinate (B) at (10.5,-0.7);
	\draw (B)   node[above left]{$\mathcal W_{F}^j$};
    
    \draw [add = -0.2 and -0.1] (V) to (A);
    \draw [add = 0.1 and -0.8,dashed] (V) to (A);
    \draw [add = -0.9 and 0.1,dashed] (V) to (A);
    
    \draw [add = -0.45 and -0.2] (V) to (B);
    \draw [add = 0.1 and -0.55,dashed] (V) to (B);
    \draw [add = -0.8 and 0.2,dashed] (V) to (B);
    
	\draw[->]  (9,-1) node[below] {the chamber that $F$ is semistable} -- (13,-1);
    \draw[->]  (9,-3) node[below] {the stability condition $\sigma_{0,0+}$} -- (13,-2);
    \draw (17,5.5) node[left] {Curve $\Gamma$};
\end{tikzpicture}
}
\end{center}
\caption{Describing walls via $\Ker\zab \subset H^*(X, \R)$.} \label{fig:wallspicture}
\end{figure}
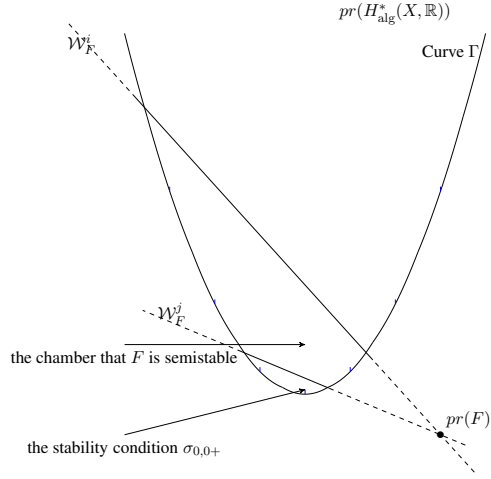

\begin{Prop}[{\cite[Proposition 9.3]{Bridgeland:K3}}] \label{prop:walls}
Fix an object $F \in D^b(X)$. There exists a collection of line segments (walls) $\cW_F^i$ in $\Gamma_+$  with the following properties:
\begin{itemize}
\item the end points of all segments are either on the curve $\Gamma$ or the on the line segment through $(n,\frac{H^2}2n^2)$ to $(n,\frac{H^2}2n^2-1)$;
\item the extension of each line segment passes through $pr(F)$ if $\rk(F) \neq 0$; otherwise it has slope $\ch_2(F)H^2/H.\ch_1(F)$;
\item the $\sab$-(semi)stability or instability of $F$ does not change when $\sab$ changes between two consecutive walls.  
\item the object $F$ is strictly $\sab$-semistable  if $(\beta,\alpha)$ is contained in one of the walls.
\item if $F$ is $\sab$-semistable in one of the adjacent chambers to a wall, then it is unstable in the other adjacent chamber. 
\end{itemize}
\end{Prop}
See Figure \ref{fig:wallspicture} for a picture and \cite{Fe17} for more details and further references. 

\section{Bounds for the dimension of global sections}
In this section, we prove the first part of Theorem \ref{main1} which introduces a new upper bound for the dimension of global sections of vector bundles on a curve over a K3 surface. We always assume $X$ is a K3 surface with Pic$(X) = \mathbb{Z}H$ and $C \in |H|$ is a smooth curve of genus $g$. We denote by $\iota:C\hookrightarrow X$ the embedding of the curve $C$ into $X$.
\subsection{The destabilizing wall for a semistable vector bundle on the curve $C$} Let $E$ be a slope semistable vector bundle on the curve $C$. By \cite[Theorem 3.11]{Maciocia:computing-the-walls}, the push-forward $\iota_*E$ is $\sigma_{\beta,\alpha}$-semistable for any $\beta \in \mathbb{R}$ and $\alpha$ sufficiently large. Suppose $\iota_*E$ becomes strictly semistable at the wall $\mathcal{W}$ which passes through  $\sigma_{0,\alpha}$ for some $\alpha>0$ and intersects the curve $\Gamma$ at $(\beta_1,\Gamma(\beta_1))$ and $(\beta_2,\Gamma(\beta_2))$ for $\beta_1< 0< \beta_2$.
\begin{Lem}\label{lem:fw}
Adopt notations as above, we have
\begin{equation*}
-1 + \frac{d}{rH^2} \leq \beta_1 \;\;\;\text{ and }\;\;\;\beta_2\leq \frac{d}{rH^2}. 
\end{equation*}
\end{Lem}
\begin{proof}
Let $0\rightarrow F_2\rightarrow \iota_*E\rightarrow F_1\rightarrow 0$ in $\Coh^0X$ be the destabilizing sequence at the wall $\cW$, then there is an exact sequence in $\Coh X$:
\begin{center}
\begin{tikzcd}[row sep=tiny]
0\arrow[r] & \HH^{-1}(F_1) \arrow[r] 
& F_2 \arrow[r] & \iota_*E \arrow[r] & \HH^{0}(F_1) \arrow[r] & 0. \\
\text{rank} & s & s & 0 & 0 \\
\ch_1 & d_1H & d_2H & rH & aH
\end{tikzcd}
\end{center}

If $s=0$, then since $F_2$ and $\iota_*E$ have the same phase with respect to $\sigma_{0,\alpha}$, we must have $\ch(\iota_*E) = k\ch(F_2)$ for some real number $k \neq 0$ and $F_2$ cannot make a wall for $\iota_*E$. Thus, we may assume $s \neq 0$. 
Let $T(F_2)$ be the maximal torsion subsheaf of $F_2$ and $\ch_1(T(F_2)) = tH$. Since $E$ is of rank $r$, to make the sequence exact at the term $\iota_*E$, we must have 
\begin{equation*}
r-a \leq \rank \big(\iota^*T(F_2)\big) + \rank\big(\iota^* F_2/T(F_2)\big) = s+t. 
\end{equation*}
Therefore, 
\begin{equation}\label{in.difference of slopes}
\dfrac{\ch_1\big(F_2/T(F_2)\big).H}{sH^2} - \dfrac{\ch_1\big(H^{-1}(F_1)\big).H}{sH^2} = \dfrac{d_2-t -d_1}{s} = \dfrac{r-a-t}{s} \leq 1.
\end{equation}
By Proposition \ref{prop:walls}, the object $F_1$ is semistable of the same phase as $\iota_*E$ along the line segment $\mathcal{W}$, in particular it is in the heart $\Coh^{\beta_1+\epsilon}X$ where $\epsilon \rightarrow 0^{+}$. Thus by definition of the tilting heart,
\begin{equation}
\dfrac{\ch_1(H^{-1}(F_1)).H}{H^2s} = \dfrac{d_1}{s} \leq \beta_1.\label{in:beta1}
\end{equation}
Therefore inequality \eqref{in.difference of slopes} implies
\begin{equation}\label{eq:4}
\dfrac{\ch_1(F_2/T(F_2)).H}{H^2s} \leq 1+\beta_1.  
\end{equation}
Similarly, Proposition \ref{prop:walls} gives
\begin{equation}
\frac{H\ch_1(F_2/T(F_2))}{H^2s} \geq \beta_2. 
\end{equation}
Therefore, $\beta_2 - \beta_1 \leq 1$. By the second property of Proposition \ref{prop:walls}, the slope of $\mathcal{W}$ as a line in the projection $pr(\HalgR)$ is
\begin{equation}
\label{eq5}\frac{\Gamma(\beta_2)-\Gamma(\beta_1)}{\beta_2-\beta_1}=\frac{H^2\ch_2(\iota_*E)}{H\ch_1(\iota_*E)}=-\frac{H^2}{2}+\frac{d}{r}.
\end{equation} It is not hard to see that $\beta_2$ (respectively $\beta_1$) reaches its maximum $\beta_2^{\max}$ (respectively minimum $\beta_1^{\min}$) when $\beta_2-\beta_1=1$. Substitute this to (\ref{eq5}), we get $$\Gamma(\beta_1^{\min}+1)-\Gamma(\beta_1^{\min})=-\frac{1}{2}+\frac{d}{rH^2}.$$ Solving the equation, we get  $\beta_1^{\min}=\frac{d}{rH^2}-1$ and $\beta_2^{\max}=\frac{d}{rH^2}$.
\end{proof}
We need the following description for the first wall in details. 
\begin{Lem}\label{lem:bound for beta1}
	Adopt notations from Lemma \ref{lem:fw}.
	\begin{enumerate}
       \item We always have $\beta_1\geq -1+\frac{1}{r}$.
		\item When $d\leq 2g-2+r$, we have $\beta_2\leq \frac{1}r$. 
        \item When $r=3$, we either have the Chern characters $\ch(F_2)=(3,H,-)$ or $\beta_1\geq -\frac{1}2$.
	\end{enumerate}
\end{Lem}
\begin{proof}
Adopt the notations as in the proof of Lemma \ref{lem:fw}.

 (\textit{a}):  We know $ \frac{\ch_1(F_2/T(F_2))H}{H^2} \geq 1$, so $\frac{\ch_1(H^{-1}(F_1))H}{H^2} \geq 1-r$. If $s = \rk(F_2)\leq r$, then $\frac{H\ch_1(F_2/T(F_2))}{H^2s} \geq \frac{1}{r}$ and by inequality (\ref{eq:4}), we have $\beta_1\geq \frac{2-r}{r}$. If $s>r$, then $\beta_1\geq \frac{\ch_1(F_1)H}{H^2s} \geq \frac{1-r}{r}$. \\  
	
	(\textit{b}):  If $\ch_1(F_2/T(F_2)) = H$ and $s = \rk(F_2) \geq r$, then $ \beta_2 \leq \frac{H\ch_1(F_2/T(F_2))}{H^2s} \leq \frac{1}{r}$ and the claim follows. Thus, we may assume that either $\ch_1(F_2/T(F_2)) = H$ with $s<r$ or $\frac{\ch_1(F_2/T(F_2))H}{H^2} \geq 2$. In any of these two cases, we first show that $\beta_1 \geq \frac{2-r}{r-1}$. 

If $\ch_1(F_2/T(F_2)) = H$ and $s = \rk(F_2) < r$, then $\beta_1 \geq \frac{\ch_1((F_2/T(F_2))H}{H^2s} \geq \frac{1}{r-1}$ which gives $\beta_1 \geq \frac{2-r}{r-1}$. If $ \frac{\ch_1(F_2/T(F_2))H}{H^2} \geq 2$, then $\frac{\ch_1(F_1)H}{H^2} \geq 2-r$. Thus, if $s = \rk(F_2)\leq r$, then $\frac{H\ch_1(F_2/T(F_2))}{H^2s} \geq \frac{2}{r}$ and by inequality (\ref{eq:4}), we have $\beta_1\geq \frac{2-r}{r}$ and if $s>r$, then $\beta_1\geq \frac{\ch_1(H^{-1}(F_1))H}{H^2s} \geq \frac{2-r}{r}$.

	Note that by (\ref{eq5}), when $\beta_1$ increases, $\beta_2$ will decrease. We only need to show that when $d\leq 2g+r$,
	$$\frac{\Gamma\left(\frac{1}r\right)-\Gamma\left(\frac{2-r}{r-1}\right)}{\frac{1}r-\frac{2-r}{r-1}}\geq -\frac{1}{2}+\frac{d}{rH^2}.$$
	The left hand side is equal to
	\begin{align*}
	& \frac{1}2\left(\frac{2-r}{r-1}+\frac{1}r\right)+\frac{1}{H^2}\frac{\frac{1}{r^2}-\frac{1}{(r-1)^2}}{\frac{1}r+\frac{r-2}{r-1}} \\
	= & -\frac{1}2+\frac{1}r+\frac{1}{2r(r-1)}-\frac{2r-1}{H^2r(r-1)(r^2-r-1)} \\
	\geq & -\frac{1}2+\frac{1}r+\frac{1}{H^2}\geq -\frac{1}{2}+\frac{d}{rH^2}.
	\end{align*}
	The first inequality in the last line holds as $r\geq 3$ when $0>\beta_1\geq \frac{2-r}{r-1}$.

 (\textit{c}): Denote $d_1=\frac{D_1H}{H^2}$ and $d_2=\frac{D_2H}{H^2}$. If  $d_2\geq 2$, then we must have $d_1= -1$. Thus by  (\ref{in.difference of slopes}), we have $s\geq 2$ and inequality (\ref{in:beta1}) implies that $\beta_1\geq -\frac{1}2$. Thus we may assume $d_2=1$. If $s<3$, then $\frac{H\ch_1(F_2)}{H^2\rk(F_2)}\geq \frac{1}{2}$ (note that we must have $\ch_1(T(F_2))=0$) and inequality (\ref{eq:4}) gives $\beta_1\geq -\frac{1}{2}$. When $s>3$, by (\ref{in:beta1}), we have $\beta_1\geq \frac{d_2}{s}\geq \frac{-2}{4}=-\frac{1}2$.
\end{proof}

\subsection{An upper bound on the dimension of global sections}
We first recall the result in \cite[Section 3]{Fe17}. Define the function $\overline{Z} \colon K(X) \rightarrow \mathbb{C}$ as 
\begin{equation*}
\overline{Z}(F) = \ch_2(F) + i \frac{\ch_1(F).H}{H^2}.
\end{equation*}
We also define the following non-standard norm on $\mathbb{C}$: 
\begin{equation*}
\lVert x+iy \rVert = \sqrt{x^2 + (2H^2+4)y^2}. 
\end{equation*} 
The next proposition bounds the dimension of global sections of objects in terms of the length of a polygon. 
 \begin{Prop}[{\cite[Proposition 3.4]{Fe17}}]\label{prop:bound for global sections} 
	Let $F \in \Coh^0X$ be an object which has no subobject $F' \subset F$ with $\ch_1(F') = 0$. 
	\begin{itemize*}
		\item[(a)] There exists $\epsilon >0 $ such that the Harder-Narasimhan filtration of $F$ is a fixed sequence
		\begin{equation*}\label{HN}
		0 = \tilde{E}_{0} \subset \tilde{E}_{1} \subset .... \subset \tilde{E}_{n-1} \subset  \tilde{E}_n=F,
		\end{equation*}
		for all stability conditions $\sigma_{0,\alpha}$ where $0< \alpha < \epsilon$.
		\item[(b)] Let $p_i \coloneqq  \overline{Z}(\tilde{E_i})$ for $0 \leq i \leq n$, then
		$$
		h^0(X,F) \leq \dfrac{\chi(F)}{2}  + \dfrac{1}{2} \sum_{i=1}^{n} \lfloor \lVert p_ip_{i-1} \rVert \rfloor
		$$
		where $\lfloor \lVert p_ip_{i-1} \rVert \rfloor$ is the integer part of the length of the line segment $p_ip_{i-1}$ and $\chi(F)$ is the Euler characteristic of $F$.
	\end{itemize*} 	
\end{Prop}
We denote by $P_F$ the polygon with the extremal points $\{p_0,p_1,...,p_n\}$ which is a convex polygon.

Let $E$ be a slope semistable rank $r$-vector bundle on the curve $C$ of degree $d$. 
Proposition \ref{prop:bound for global sections} implies that there exists $\epsilon >0$ such that the Harder-Narasimhan filtration of $\iota_*E$ with respect to the stability condition $\sigma_{0,\alpha}$ for positive $\alpha<\epsilon$ is a fixed sequence 
\begin{equation*}
0 = \tilde{E}_{0} \subset \tilde{E}_{1} \subset .... \subset \tilde{E}_{n-1} \subset  \tilde{E}_n=\iota_*E.
\end{equation*}
Consider the triangle $opq$ where $o$ is the origin, $q = \overline{Z}(\iota_*E)$, the slope of $\overline{op}$ is equal to $\beta_2/ \Gamma(\beta_2)$ and the slope of $\overline{pq}$ is $\beta_1/\Gamma(\beta_1)$, where the real numbers $\beta_1$ and $\beta_2$ are defined as in Lemma \ref{lem:fw}.
	\begin{figure} [h]
	\begin{centering}
		\definecolor{zzttqq}{rgb}{0.27,0.27,0.27}
		\definecolor{qqqqff}{rgb}{0.33,0.33,0.33}
		\definecolor{uququq}{rgb}{0.25,0.25,0.25}
		\definecolor{xdxdff}{rgb}{0.66,0.66,0.66}
		\scalebox{0.8}{
		\begin{tikzpicture}[line cap=round,line join=round,>=triangle 45,x=1.0cm,y=1.0cm]

		\draw[->,color=black] (0,0) -- (0,3.5);
		\draw[->,color=black] (-3,0) -- (3,0);
		\draw[color=black] (0,0) -- (-3,3.3);
		\draw[color=black] (0,0) -- (1.8,1.5);
		\draw[color=black] (1.8,1.5) -- (-3,3.3);
		
		\draw[color=black] (0,0) -- (.5,1);
		\draw[color=black] (-.5,2) -- (-3,3.3);
		\draw[color=black] (-.5,2) -- (.5,1);
		
		\draw (0,0) node [below] {$o$};
		\draw (3,0) node [right] {$\ch_2(-)$};
		\draw (0,3.5) node [above] {$\dfrac{\ch_1(-).H}{H^2}$};
		\draw (1.8,1.5) node [right] {$p$};
		\draw (-3,3.3) node [above] {$q=\overline{Z}(\iota_*E)$};
		\draw (.5,1) node [right] {$p_1$};
		\draw (-.5,2) node [below] {$p_2$};
		
		\begin{scriptsize}
		
		\fill [color=black] (0,0) circle (1.1pt);
		\fill [color=black] (-3,3.3) circle (1.1pt);
		\fill [color=black] (1.8,1.5) circle (1.1pt);
		\fill [color=black] (.5,1) circle (1.1pt);
		\fill [color=black] (-.5,2) circle (1.1pt);

		\end{scriptsize}
		
		\end{tikzpicture}
		}
		\caption{The polygon $P_{\iota_*E}$ is inside the triangle $opq$}
		
		\label{wall.4}
		
	\end{centering}
	
\end{figure}
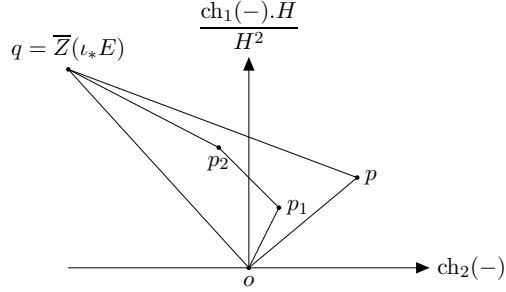
\begin{Lem}\label{lem:polygon} \label{lem:ph}
	The polygon $P_{\iota_*E}$ is contained in the triangle $opq$. 	
\end{Lem}
\begin{proof}
If $\iota_*E$ is $\sigma_{0,\alpha}$-semistable where $\alpha \rightarrow 0^{+}$, then the polygon $P_{\iota_*E}$ is just the line segment $oq$ and the claim follows. Thus, we may assume $\iota_*E$ is not $\sigma_{0,\alpha}$-semistable where $\alpha \rightarrow 0^{+}$. Since the polygon $P_{\iota_*E}$ is convex, it suffices to show that 
\begin{equation*}
\frac{H^2\ch_2(\tilde{E}_1)}{H\ch_1(\tilde{E}_1)} \leq \dfrac{\Gamma(\beta_2)}{\beta_2} \;\;\;\; \text{and} \;\;\;\;  \dfrac{\Gamma(\beta_1)}{\beta_1}  \leq \frac{H^2\ch_2(\iota_*E/\tilde{E}_{n-1})}{H\ch_1(\iota_*E/\tilde{E}_{n-1})}. 
\end{equation*} 
The phase of the subobject $\tilde{E}_1$ in the Harder-Narasimhan filtration is bigger than phase of $\iota_*E$ at the stability condition $\sigma_{0,\alpha}$ where $\alpha \rightarrow 0^+$. Therefore there are stability condition between large volume limit ($\sigma_{\beta,\alpha}$ where $\alpha \rightarrow \infty$) and the stability conditions $\sigma_{0,\alpha}$ where $\alpha \rightarrow 0^{+}$ such that $\tilde{E}_1$ and $\iota_*E$ have the same phase. Proposition \ref{prop:walls} implies that these stability conditions are on a line segment $L$ whose extension passes through the point $pr(\tilde{E}_1)$. Note that $\rk(\tilde{E}_1) \neq 0$. By assumption, the line $L$ is lower than the wall $\mathcal{W}$ for $\iota_*E$, see Figure \ref{fig:phase}.

	\begin{figure}[h]
\begin{center}
\scalebox{0.8}{
\tikzset{%
    add/.style args={#1 and #2}{
        to path={%
 ($(\tikztostart)!-#1!(\tikztotarget)$)--($(\tikztotarget)!-#2!(\tikztostart)$)%
  \tikztonodes},add/.default={.2 and .2}}
}
\begin{tikzpicture}[scale = 12]
\newcommand\XA{0.1}
\newcommand\obe{-0.3}
\newcommand\xa{-0.9}
\newcommand\xb{0.6}

\coordinate (A) at (\xa/4,\xa*\xa/2-0.1+0.1+0.2*\xa+0.1*\xa*\xa);
\node [above right] at (A) {$(\beta_1,\Gamma(\beta_1))$};
\node at (A) {$\bullet$};

\coordinate (B) at (\xb/4,\xb*\xb/2-0.1+0.1-0.2*\xb+0.1*\xb*\xb);
\node [above right] at (B) {$(\beta_2,\Gamma(\beta_2))$};
\node at (B) {$\bullet$};

\newcommand\xc{0.5}
\coordinate (C) at (\xc/4,\xc*\xc/2-0.1+0.1*\xc*\xc);
\node at (C) {$\bullet$};

\newcommand\xd{-0.8}
\coordinate (D) at (\xd/4,\xd*\xd/2-0.1+0.1+0.2*\xd+0.1*\xd*\xd);
\node at (D) {$\bullet$};

\draw [add =0 and 0] (B) to (A);

\draw [add =0 and -0.39] (D) to (C);
\draw [add =-0.61 and 0] (D) to (C);

\draw [add =-1 and 0.3,dashed] (D) to (C);

\draw [add =-1 and 0.1,dashed] (D) to (C) node{$\bullet$} node [above right]{$pr(\tilde{E}_1)$} coordinate (E);

\draw [add = 0 and 1.2, dashed] (0,0) to (E) node[right]{slope $=\frac{H^2\ch_2(\tilde{E}_1)}{H\ch_1(\tilde{E}_1)}$};

\draw [add = 0 and 1, dashed] (0,0) to (B) node[right]{slope $=\frac{\Gamma(\beta_2)}{\beta_2}$};

\draw (-0.1,0.19) node{$L$};
\draw (-0.1,0.26) node{$\mathcal{W}$};

\draw[->] [opacity=\XA*5] (-0.3,0) -- (0,0) node[above right] {}-- (0.4,0) node[below right, opacity =1] {$\frac{H\ch_1}{H^2\rk}$};
\draw (-0.25,0) to (-0.25,0.01);
\node [below] at (-0.25,0) {$-1$};
\draw (0.25,0) to (0.25,0.01);
\node [below] at (0.25,0) {$1$};
\draw[->][opacity=\XA*5] (0,-.2)-- (0,0) node [above right] {O} --  (0,0.45) node[right, opacity=1] {$\frac{\ch_2}{\rk}$};
\draw (0,0.4) to (-0.004,0.4);
\node [right] at (0,0.4) {$0.4H^2$};
\draw [opacity=\XA](-0.25,0.5) parabola bend (0,0) (0.25,0.5);
\draw [opacity=\XA](-0.25,0.4) parabola bend (0,-0.1) (0.25,0.4);

\draw[domain=-0.125:0.125,smooth,variable=\t,blue] plot ({\t},{8*\t*\t-0.1+1.6*\t*\t});
\draw[domain=-0.25:-0.125,smooth,variable=\t,blue] plot ({\t},{8*\t*\t+0.8*\t+1.6*\t*\t});
\draw[domain=0.125:0.25,smooth,variable=\t,blue] plot ({\t},{8*\t*\t-0.8*\t+1.6*\t*\t});
\draw[blue] (0,-0.1) to (0,0);
\draw[blue] (-0.25,0.49) to (-0.25,0.5);
\draw[blue] (0.25,0.49) to (0.25,0.5) node[below right] {curve $\Gamma$};

\end{tikzpicture} 
}
\end{center}
\caption{Comparing slopes} \label{fig:phase}
\end{figure}
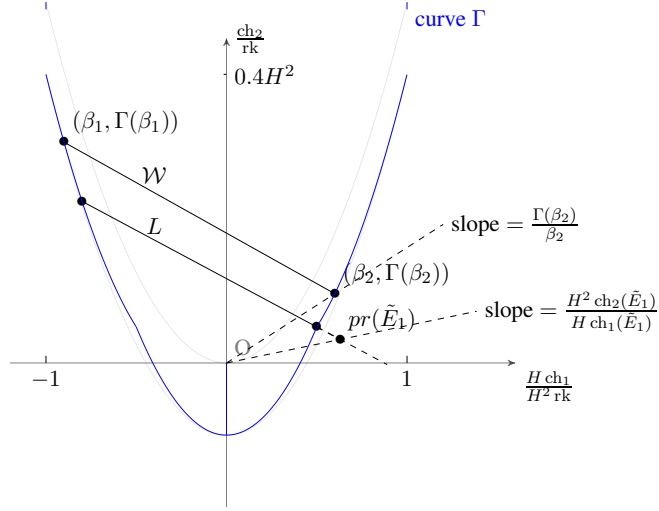
Since $\tilde{E}_1$ is $\sigma_{0,\alpha}$-semistable where $\alpha \rightarrow 0^{+}$, the point $pr(\tilde{E}_1)$ is not in $\Gamma_+$. Therefore, $pr(\tilde{E}_1)$ is on the dashed part of the line $L$ and the first claim follows. By a similar argument one can show the second claim for $\iota_*E/\tilde{E}_{n-1}$.  
\end{proof}
We are now ready to prove the bound for the dimension of global sections of the vector bundle $E$.
\begin{proof}[Proof for the first part of Theorem \ref{main1}] 
	Consider the triangle $op'q$ where the slope of $op'$ is $\frac{\frac{d}{rH^2}}{\Gamma\left(\frac{d}{rH^2}\right)}$ and the slope of $p'q$ is $\frac{\left(\frac{d}{rH^2}-1\right)}{\Gamma\left(\frac{d}{rH^2}-1\right)}$. Lemma \ref{lem:fw} implies that the triangle $opq$ is inside the triangle $op'q$, so by Lemma \ref{lem:polygon} the polygon $P_{\iota_*E}$ is also inside the triangle $op'q$. By a direct computation, one can show that the point
	\begin{equation*}
	p'= \bigg(\dfrac{d^2g}{(H^2)^2r} -r, \dfrac{d}{H^2}  \bigg).
	\end{equation*}
	Now Proposition \ref{prop:bound for global sections}, part $(b)$ gives 
	\begin{align*}
	&h^0(C,E) = h^0(X,\iota_*E) \leq \dfrac{\chi(\iota_*E)}{2} + \dfrac{1}{2} \sum_{i=1}^{n} \lVert p_ip_{i-1} \rVert \\
	\leq & \dfrac{\chi(\iota_*E)}{2} + \dfrac{1}{2}\big(   \lVert op' \rVert  + \lVert p'q \rVert \big) \\
	= & \frac{1}2(r(1-g)+d)+ \frac{1}2 \sqrt{\left(\frac{d^2g}{(H^2)^2r}-r\right)^2+4g\left( \frac{d}{H^2}\right)^2}\\
	& + \frac{1}2 \sqrt{\left(r(g-1)-d+\frac{d^2g}{(H^2)^2r}-r\right)^2+4g\left( r - \frac{d}{H^2}\right)^2} \\
	= & \frac{1}2(r(1-g)+d)+ \frac{1}2 \left(\frac{d^2g}{(H^2)^2r}+r\right)+\frac{1}2 \left(r(g-1)-d+\frac{d^2g}{(H^2)^2r}+r + \delta \right).
	\end{align*}
	One can easily show that $\delta <\dfrac{2r}{g}$ and the claim follows. 
\end{proof}
\begin{Rem}\label{rem:sharp}
	The bound for $h^0(C,E)$ in Theorem \ref{main1} is not far from the sharp bound. Let $k$ be an integer in $[1,r]$, denote $t=\gcd(r,k)$. When $d=2k(g-1)$ such that $g\geq \left(\frac{r}{t}\right)^2$, there exists a stable vector bundle $F$ on $X$ with Chern characters: $$(\rk(F),\ch_1(F),\ch_2(F))=\left(\frac{r}{t}, \frac{k}{t} H, \lf\frac{t}r+\frac{k^2}{rt}(g-1) \rf - \frac{r}{t} \right).$$ 
	The restriction $F^{\oplus t}|_C$ is semistable (by \cite[Theorem 1.1]{Fe16}) with rank $r$, degree $2k(g-1)$ and dimension of global sections 
	\begin{align*}
	    h^0(C,F^{\oplus t}|_C)= h^0(X,F^{\oplus t}) = t\lf \frac{t}r+\frac{k^2}{rt}(g-1) \rf + r.
	\end{align*}
	If the $\lfloor\cdot\rfloor$ function can be dropped for free, the formula can be simplified as \[r+\frac{g}{4r(g-1)^2} d^2 +\frac{t^2 -k^2}{r}.\] 
\end{Rem}

\begin{Cor}\label{cor:clf}
	Let $(X,H)$ be a polarized smooth K3 surface with $\Pic(X)=\Z H$. Let the smooth curve $C\in|H|$ be with genus $g$, $E$ be a semistable vector bundle with rank $r$ and degree $d$ $(0\leq d\leq r(g-1)).$ Then $\clf(E)> \frac{d}r-\frac{d^2g}{2r^2(g-1)^2}-\frac{2}g$ and \[\clf_r(C)> 2\sqrt{g-1}-2 -\frac{2\sqrt{g-1}}{g}.\]
\end{Cor}
\begin{proof}
The bound for $\clf(E)$ is by substituting the bounds of $h^0(C,E)$ into the formula of Clifford index. By the first part of Theorem \ref{main1}, if $h^0(C,E)\geq 2r$, then $d>\frac{2r(g-1)^{\frac{3}2}}{g}$. When $d\in \left[\frac{2r(g-1)^{\frac{3}2}}{g},r(g-1)\right]$ and $g\geq 3$, the function $f(d)=\frac{d}r-\frac{d^2g}{2r^2(g-1)^2}-\frac{2}g$ reaches its minimum at the left boundary. Therefore, $$\clf_r(C)>\frac{2r(g-1)^{\frac{3}2}}{gr}-\frac{4r^2(g-1)^3}{2r^2g(g-1)^2}-\frac{2}g = 2\sqrt{g-1}-2 -\frac{2\sqrt{g-1}}{g}$$ for any $r$.
\end{proof}

\section{Higher rank Clifford indices}
In this section, we compute higher rank Clifford indices of curves over K3 surfaces and prove the second part of Theorem \ref{main1}.
\subsection{Picard number one case}
We assume $X$ is a K3 surface with Pic$(X) = \mathbb{Z}H$ and $C \in |H|$ is a smooth curve of genus $g$. Denote by $\iota:C\hookrightarrow X$ the embedding of the curve $C$ into $X$. We first briefly recall the main result in \cite{Fe16}, which constructs semistable vector bundles on $C$ by restricting vector bundles on $X$ with low discriminant. By \cite[Theorem 2.15]{BM:walls}, there exists a slope stable sheaf $\tilde{E}_r$ on $X$ with Chern character $(r,H,\fgr-r)$. Define $E_r:=\tilde{E}_r|_C$.

\begin{Thm}[{\cite[Theorem 1.1]{Fe16}}]
\label{thm:constofer}
Assume $g\geq \max\{r^2,6\}$ and $r\geq 2$, then the sheaf $E_r$ is a semistable vector bundle on $C$ with $h^0(C,E_r)\geq 2r$ and 
\begin{align}
\clf(E_r) \leq \frac{2}r(g-1)-\frac{2}r\fgr. 
\label{eq:clflower}
\end{align}
\end{Thm}
\begin{proof}
We first check $\tilde{E}_r$ satisfies conditions in \cite[Theorem 1.1]{Fe16}. If $\tilde{E}_r$ is not locally free, then the double dual $F=\tilde{E}_r^{\vee\vee}$ is slope stable with Chern characters $(r,H,\fgr-s)$ for some integer $s\leq r-1$. Yet $-\chi(F,F)=H^2-2r\left(\fgr-s\right)-2r^2<-2$. This contradicts \cite[Theorem 2.15]{BM:walls}. By the assumption on $r$ and $g$, we have 
\[H^2+\frac{H^2(r-2)}{(r-1)^2}-2r^2-\left(H^2-2r\fgr\right)>0.\]
Therefore \cite[Theorem 1.1]{Fe16} implies that $E_r$ is slope semistable on $C$.

Apply $\Hom(\cO_X,-)$ to the short exact sequence: $0\rightarrow \tilde{E}_r(-H)\rightarrow \tilde{E}_r\rightarrow \iota_*E_r\rightarrow 0.$ 
Since $\tilde{E}_r(-H)$ is slope stable and $\mu_H(\tilde{E}_r(-H))<\mu_H(\cO_X)$, we  have Hom$(\cO_X,\tilde{E}_r(-H))=0$. Therefore, 
\[h^0(C,E)=\hom(\cO_X,\iota_*E_r)\geq \hom(\cO_X,\tilde E_r) = \chi(\tilde E_r) = \fgr+r \geq 2r.\]
As $\deg(E_r)=c_1(\tilde E_r)H = 2(g-1)$, by a direct computation, $\clf(E_r) \leq \frac{2}r(g-1)-\frac{2}r\fgr$.
\end{proof}
We now prove the Clifford index of $E_r$ is indeed the minimum of Clifford index of any semistable vector bundle $E$ with rank $r$, degree $d$ and $h^0(E)\geq 2r$. This will involve several different cases.  
\begin{proof}[Proof of the second part of Theorem \ref{main1} for $r \geq 4$]
Let $E$ be a semistable rank $r$-vector bundle of degree $d \leq r(g-1)$ on the curve $C$. By Theorem \ref{thm:constofer}, it suffices to show that either $h^0(E) <2r$ or $\clf(E) \geq  \frac{2}{r}(g-1)-\frac{2}r\fgr$.\\ 

\textbf{Step 1.} We show $\clf(E) >  \frac{2}{r}(g-1)-\frac{2}r\fgr$ if $2g+2<d\leq r(g-1)$.
	
	Denote $t \coloneqq d-2(g-1)$. The first part of Theorem \ref{main1} implies that
\begin{align*}
& \clf(E)-\frac{2}{r}(g-1)+\frac{2}r\fgr \\
> & \frac{t}r-\frac{2}r\left(r+\frac{\left(2+\frac{t}{g-1}\right)^2}{4r}g+\rg\right)+2+\frac{2}r\fgr \eqqcolon Q(t).
\end{align*}
Then $Q(t)$ is a quadratic function with respect to $t$ with negative leading coefficient. Thus it suffices to show that $Q(t=5) >0$ and $Q(t= (r-2)(g-1)) >0$ which can be easily checked by direct computations.\\  

\textbf{Step 2.} We show $\clf(E) \geq  \frac{2}{r}(g-1)-\frac{2}r\fgr$ if $-2 \big(\fgr -r \big) \leq d-2(g-1) \leq 4$.

Applying Proposition \ref{prop:bound for global sections} for the push-forward $\iota_*E$ implies that there exists $\epsilon >0$ such that its Harder-Narasimhan filtration with respect to $\sigma_{0,\alpha}$ for positive $\alpha<\epsilon$ is a fixed sequence 
\begin{equation*}
0 = \tilde{E}_0 \subset \tilde{E}_1 \subset ... \subset \tilde{E}_{n-1} \subset \tilde{E}_n = \iota_*E,
\end{equation*}   
and
\begin{equation}\label{lower bound for cliff}
\clf(E) \geq g+1 - \dfrac{l(E)}{r},
\end{equation} 
where $l(E) \coloneqq \sum_{i=1}^{n}\lfloor \lVert p_ip_{i-1} \rVert \rfloor $ and $p_i = \overline{Z}(\tilde{E}_i)$. Thus it is suffices to show that
\begin{equation}\label{in.final}
l(E) \leq g(r-2) +2 \fgr +r+2. 
\end{equation}
We first treat with the case that $\ch_1(\tilde{E}_1)\neq H$. In this case, as $\frac{H\ch_1(\tilde{E}_1)}{H^2}$ is a positive integer, $\frac{H\ch_1(\tilde{E}_1)}{H^2}\geq 2$. By Lemma \ref{lem:bound for beta1}, $\beta_1 \geq -1+1/r$. Applying the same argument as in Lemma \ref{lem:polygon} implies that the polygon $P_{\iota_*E}$ is contained in the triangle $op'q$ where the slope of $qp'$ is $\frac{-1+1/r}{\Gamma(-1+1/r)}$ and the vertical coordinate of the point $p'$ is equal to $2$, see Figure \ref{fig:mainp}.
\begin{figure}[h]
	\begin{center}
		\tikzset{%
			add/.style args={#1 and #2}{
				to path={%
					($(\tikztostart)!-#1!(\tikztotarget)$)--($(\tikztotarget)!-#2!(\tikztostart)$)%
					\tikztonodes},add/.default={.2 and .2}}
		}
		\scalebox{0.7}{
			\begin{tikzpicture}            
			\coordinate (O) at (0,0);
			\draw (O) node {$\bullet$} node[below] {\textbf{O}};
			
			\coordinate (C) at (0.6,1);
			\draw[dashed] (0,1) node[left] {$1$} to (C);
			\coordinate (D) at (1,1);\draw (D) node {$\bullet$};
			
			\coordinate(F) at (-1.9,2);
			\draw[dashed] (0,2) node[right] {$2$} to (F);
			\draw (F) node {$\bullet$} node[above] {$p'$};
			
			\coordinate (E) at (-12,5.5);
			\draw (E) node {$\bullet$} node[left] {$q$};
			\draw[dashed] (0,5.5) node[right] {$r$} to (E);
			\draw[dashed] (-12,0) node[below] {$d -r(g-1)$} to (E);
			
			\draw[->] (-14,0) to (4,0) node[right] {$\ch_2$};
			\draw[->] (0,0) to (0,6) node[above] {$\frac{H\ch_1}{H^2}$};
			
			\draw[dashed] (D) node[right] {$\tilde{p}$} to (E);
			\draw (O) to (E);
			\draw (O) to (C) node {$\bullet$} node[above] {$p_1$} to (F) to (E);

			\draw (-4.5,4) node {slope of $q\tilde{p} =  \frac{ \frac{1-r}r}{\Gamma\left(\frac{1-r}r\right)}$};
			\end{tikzpicture}
		}
	\end{center}
	\caption{The polygon $p_{\iota_*E}$ is inside the polygon $op_1p'q$.} \label{fig:mainp}
\end{figure}
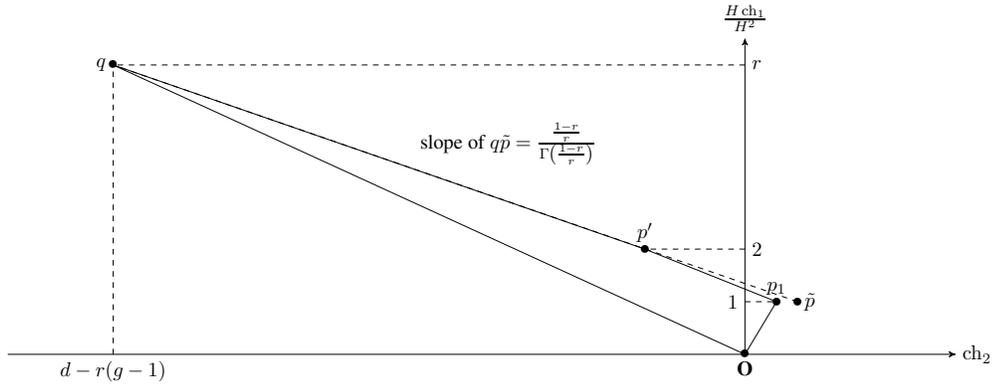

Denote by $\tilde{p}$ the point along the line $p'q$ with the vertical coordinate equal to $1$. The coordinates of two points $p'$ and $\tilde{p}$ are
\begin{equation*}
p'= \bigg( d-2(g-1) -\dfrac{r-2}{r}(g+r) \, ,\, 2\bigg) \;\;\; \text{and} \;\;\; \tilde{p} = \bigg(d-2(g-1) + \dfrac{g}{r} -r \,, \, 1\bigg).
\end{equation*} 
Note that the length $\lVert q\tilde{p} \rVert$ does not depend on $d$,
\begin{align}
\lVert p'q \rVert & = \frac{r-2}{r-1} \lVert \tilde{p}q \rVert =\frac{r-2}{r-1}  \sqrt{\left((r-2)(g-1)+\gr-r\right)^2+ 4g(r-1)^2} \\
& < \frac{r-2}{r-1} \left((r-2)(g-1)+\gr+r + \frac{2r}{g}\right)\label{ED}
\end{align}
The horizontal coordinate of $p'$ is negative and is bigger than $-g+r+2$. Thus if $r \geq 4$, we have 
$$\lVert op' \rVert \leq \sqrt{16g+(g-r-2)^2}\leq g+r+\frac{1}r-\frac{6}5.$$
This implies $ l(E) \leq \lfloor \lVert op' \rVert + \lVert p'q \rVert \rfloor \leq g(r-2)+2\fgr +r+2 $, so inequality \eqref{in.final} holds.

Now assume $\ch_1(\tilde{E}_1) = H$. By Lemma \ref{lem:bound for beta1}, we have $\beta_2\leq \frac{1}r$. Therefore, $\ch_2(\tilde{E}_1)\leq \lfloor\frac{\Gamma(\beta_1)}{\beta_1}\rfloor=\fgr-r$. 
Write $s \coloneqq -\ch_2(\tilde{E}_1) - \fgr +r$, thus $s \geq -2\fgr + 2r$. Define the function 
\begin{equation}\label{funcion f}
f(s) \coloneqq \sqrt{4g +  \bigg(g \frac{r-2}{r} +  \fgr  -s -2\bigg)^2  }.
\end{equation}
One can easily show that 
\begin{equation}\label{in.s}
f(s) < g \frac{r-2}{r} +  \fgr  -s + \frac{2}{r-1} - \dfrac{r-2}{r-1}\frac{2r}{g} + \delta,
\end{equation}
where $\delta = 1$ if $s=0$ and $\delta = 2$ if $s \in \bigg[1,\dfrac{g}{2r}\bigg]$. We consider three different cases:\\

Case I: when $s \leq 0$, in particular $|\ch_2(\tilde{E}_1)| \leq \fgr -r$, the length $\lVert op_1 \rVert$ is maximum if $\ch_2(\tilde{E}_1) = -\fgr +r $, recall that $p_1 = \overline{Z}(\tilde{E}_1)$.
\begin{equation*}
\lfloor \lVert op_1 \rVert \rfloor \leq \left\lfloor \sqrt{4g+\left(\fgr-r\right)^2}\right\rfloor = \fgr+r.
\end{equation*} 
In this case, the point $p'$ is on the left hand side of $p_1$, so the length of $p_1p'$ is maximum when $d$ is minimum and the horizontal coordinate of $p_1$ is maximum, i.e. $d-2(g-1) + \fgr -r$, thus 
\begin{equation*}
\lVert p_1p' \rVert \leq \sqrt{4g+\bigg(g\dfrac{r-2}{r} + \fgr -2 \bigg)^2} = f(s=0). 
\end{equation*}
Inequalities \eqref{ED} and \eqref{in.s} for $s=0$ imply that 
\begin{equation*}
\lfloor \lVert p_1p' \rVert + \lVert p'q \rVert   \rfloor\leq     (r-2)(g-1)+\fgr + 2.
\end{equation*}
Therefore, $l(E) \leq \lfloor \lVert op_1 \rVert \rfloor + \lfloor \lVert p_1p' \rVert + \lVert p'q \rVert  \rfloor \leq (r-2)(g-1) + 2\fgr +r+2$, so again inequality \eqref{in.final} holds. \\

Case II. when $1 \leq s \leq \dfrac{g}{2r}$, then 
\begin{align*}
\lVert op_1 \rVert  =  \sqrt{4g+\left(\fgr+s -r\right)^2} < \fgr +s+r 
\end{align*}
The point $p'$ is still on the left hand side of the point $p_1$, so the length of $p_1p'$ is maximum when $d$ is minimum. Therefore, $\lVert p_1p' \rVert \leq f(s)$. Combining inequalities \eqref{ED} and \eqref{in.s} implies that
\begin{equation}\label{in.sum}
\lVert p_1p' \rVert + \lVert p'q \rVert < (r-2)(g-1)+\fgr -s +4. 
\end{equation} 
 Thus
 \begin{equation*}
l(E) \leq \lfloor \lVert op_1 \rVert \rfloor + \lfloor \lVert p_1p' \rVert + \lVert p'q \rVert  \rfloor \leq (r-2)(g-1)+2\fgr + r+2.
 \end{equation*}
as we required.\\ 

Case III. when $\dfrac{g}{2r} \leq s$, the summation of lengths $ \lVert op_1 \rVert + \lVert p_1p' \rVert + \lVert p'q \rVert $ is maximum when $d$ is minimum and $\ch_2(\tilde{E}_1)$ is maximum, i.e. $s=\dfrac{g}{2r}$. In this case, 
\begin{equation*}
\lVert op_1 \rVert = \sqrt{4g + \bigg(\fgr+ \frac{g}{2r} -r\bigg)^2} < \fgr + \dfrac{g}{2r} +r- 1
\end{equation*}
Together with inequality \eqref{in.sum} for $s=\dfrac{g}{2r}$, we have 
\begin{equation*}
\lVert op_1 \rVert + \lVert p_1p' \rVert + \lVert p'q \rVert < (r-2)(g-1)+2\fgr + r +3
\end{equation*}
Since $l(E) \leq \lfloor \lVert op_1 \rVert + \lVert p_1p' \rVert + \lVert p'q \rVert  \rfloor$ inequality \eqref{in.final} is satisfied.\\

\textbf{Step 3.} We show $h^0(C,E) <2r$ if $d < 2(g-1) -2 \left(\fgr-r\right)$.\\
	By using the same notations as in Step $2$, we first consider the case $\ch_1(\tilde{E}_1) \neq H$. By Proposition \ref{prop:bound for global sections}, it suffices to show that  
	\begin{equation*}
	k(d) \coloneqq d+r(1-g) + \lVert op' \rVert + \lVert p'q \rVert < 4r.
	\end{equation*}
	One can easily check that the function 
	\begin{equation*}
	k(d) = d+r(1-g) + \sqrt{16g + \bigg(d-2(g-1) - \dfrac{r-2}{r}(g+r)\bigg)^2} + \lVert p'q \rVert
	\end{equation*}
	is increasing with respect to $d$, so
	\begin{equation*}
	k(d) \leq k\big(2(g-1)-2\big(\fgr -r\big) -1 \big) < 4r.  
	\end{equation*}
	The last inequality comes from inequality \eqref{ED} and some direct computations. Thus we may assume $\ch_1(\tilde{E}_1) = H$. Define $t \coloneqq \tilde{p}_{(x)} - \ch_2(\tilde{E}_1) -\frac{g}{r} + \fgr $, where $\tilde{p}_{(x)} = d-2(g-1)+g/r-r$ is the horizontal coordinate of the point $\tilde{p}$. We consider two different cases: \\
	
	Case I. when $0 \leq t < \frac{g}{2r}$, 
	\begin{equation}\label{t}
  \lVert op_1 \rVert = \sqrt{4g + \big(-d+2(g-1) - \fgr +t +r\big)^2} < -d+2(g-1) -\fgr +t +3r
	\end{equation} 
	Moreover, $\lVert p_1p' \rVert = f(t)$ defined in \eqref{funcion f}. Thus combining inequality \eqref{in.s} for $t \in [0,\dfrac{g}{2r}]$ and inequality \eqref{ED} give
	\begin{equation*}
	l(E) \leq \lf \lVert op_1 \rVert \rf + \lf \lVert p_1p' \rVert + \lVert p'q \rVert \rf < 4r+r(g-1) -d.   
	\end{equation*}	
	Thus the claim follows by Proposition \ref{prop:bound for global sections}. \\
	
	Case II. when $\frac{g}{2r} \leq t$, the summation of the length $\lVert op_1\rVert + \lVert p_1p' \rVert + \lVert p'q \rVert$ is maximum when $t=g/r$ for which 
	\begin{equation*}
	\lVert op_1 \rVert < -d+2(g-1)-\fgr + \dfrac{g}{r} +3r-1.
	\end{equation*}
	Since $\lVert p_1p' \rVert = f(t)$, inequality \eqref{ED} for $t=\frac{g}{2r}$ together with inequality \eqref{in.s} imply that
	\begin{equation*}
	l(E) \leq \lf \lVert op_1 \rVert + \lVert p_1p' \rVert + \lVert p'q \rVert \rf \leq 4r+r(g-1) -d,   
	\end{equation*} 
	so the claim follows.
\end{proof}

\begin{proof}[Proof of the second part of Theorem \ref{main1} for $r = 3$]
Let $E$ be a rank $3$-semistable vector bundle on the curve $C$ of degree $d$. By Lemma \ref{lem:bound for beta1}, either $\beta_1\geq -\frac{1}2$ or $\ch(F_2)=(3,H,-)$.

Case I: If $\beta_1\geq -\frac{1}2$, since $\ch_2(\iota_* E) \leq 0$, the slope of the wall $\mathcal{W}$ for $\iota_*E$ is negative. Therefore, $\abs{\beta_2} < \abs{\beta_1} \leq \frac{1}2$. Lemma \ref{lem:polygon} implies that for each of the semistable factors $\tilde{E}_i/\tilde{E}_{i-1}$ in the Harder-Narasimhan filtration of $\iota_*E$ with respect to $\sigma_{0,\alpha}$ for positive $\alpha  <\epsilon$, we have 
\begin{equation*}
\abs{\dfrac{H^2\ch_2(\tilde{E}_i/\tilde{E}_{i-1})}{H\ch_1(\tilde{E}_i/\tilde{E}_{i-1})}} \leq \dfrac{\Gamma(1/2)}{1/2}.
\end{equation*}
Therefore $l(E) \leq \left \lfloor 3\sqrt{4g + (g/2-2)^2} \right\rfloor = \lfloor 3(g/2+2) \rfloor$ and Proposition \ref{prop:bound for global sections} implies that
\[
\clf(E) \geq   g+1- \dfrac{l(E)}{3} = g-\dfrac{1}{3}\lf \frac{3g}{2}\rf -1 \geq \frac{2}3(g-1) - \frac{2}3\lf\frac{g}3\rf.
\]

Case II: If $\ch(F_2)=(3,H,-)$, then $\beta_2 \leq \frac{\ch_1(F_2/T(F_2)) .H}{3H^2} \leq \frac{1}{3}$. When $d \geq 2(g-1) - 2\big(\lf \frac{g}{3}\rf-3\big)$, define $s \coloneqq -\ch_2(\tilde{E}_1) - \lf \dfrac{g}{3} \rf +3$, then using the same argument as in Step 2 for $r \geq 4$, if $\ch_1(\tilde{E_1}) \neq H $, then for $g\ \geq 9$ and $g \neq 11$, we have  
\begin{align}
l(E) = & \lfloor \lVert op' \rVert + \lVert p'q \rVert \rfloor = \left\lfloor 2\sqrt{\bigg(\frac{g}{3}-1\bigg)^2 +g} \;+\; \sqrt{(d-\frac{7}3g+1)^2 +16g}  \right\rfloor \label{eq:le}\\
\leq & \left\lfloor 2\sqrt{\bigg(\frac{g}{3}-1\bigg)^2 +g} \;+\; \sqrt{(g-5)^2 +16g}  \right\rfloor \leq g+2 \left\lfloor \frac{g}{3} \right\rfloor +5 , \nonumber
\end{align}
which shows inequality \eqref{in.final} holds for $r = 3$. The only remaining case that the last inequality does not hold is when $g=11$, but the formula (\ref{eq:le}) is less than or equal to $22$. Therefore, we may assume $\ch_1(\tilde{E}_1) = H$. Now the arguments in Step 2, Case I, II, and III in the proof of Theorem \ref{main1} for $r \geq 4$, are valid for $r=3$, thus $\clf_3(E) \geq \frac{2}3(g-1)-\frac{2}3\left\lfloor\frac{g}3\right\rfloor$. \\
When $d < 2(g-1) - 2\big(\lf \frac{g}{3}\rf -3 \big)$, define $t = \tilde{p}_{(x)} -\ch_2(\tilde{E}_1) - \dfrac{g}{3} + \lf \dfrac{g}{3}\rf$, then again the computations in Step 3, Case I, II are valid for $r=3$, hence $h^0(F) < 6$. Therefore, the second part of Theorem \ref{main1} for $r=3$ and $g \geq 9$ follows by Theorem \ref{thm:constofer}.  
\end{proof}


\begin{proof}[Proof of the second part of Theorem \ref{main1} for $r = 2$]
Let $E$ be a semistable rank $2$-vector bundle on the curve $C$. Assume there exists a wall $\mathcal{W}$ for $\iota_*E$ and $0\rightarrow F_1\rightarrow \iota_*E\rightarrow F_2\rightarrow 0$ is the destabilizing sequence. By Lemma \ref{lem:bound for beta1}, we may assume $\ch(F_1)=(2,H,s)$ and $\ch(F_2)=(-2,H,2(1-g)+d - s)$. Since $\ch_1(F_i).H/H^2=1$ is minimal, both objects $F_1$ and $F_2$ are $\sigma_{0,\alpha}$-stable for any $\alpha >0$. Therefore, $F_1$ and $F_2$ are the Harder-Narasimhan factors of $\iota_*E$ with respect to $\sigma_{0,\alpha}$ where $0 <\alpha \ll 1$. By \cite[Theorem 2.15]{BM:walls}, 
\[-\chi(F_2,F_2)=H^2+4(2(1-g)+d - s)-8\geq -2 \Longrightarrow s\leq d-\frac{3g}2.\]
Since $F_1$ destabilizes $\iota_*E$, we have \[\frac{s}{H^2}>\frac{\ch_2(\iota_*E)}{H\ch_1(\iota_*E)}\Longrightarrow s> \frac{d}2-g+1.\]
In particular, $|s|\leq \frac{g}2-2$. Proposition \ref{prop:bound for global sections} implies that
\begin{align*}
h^0(C,E) & \leq -g+1+\frac{d}{2}+\frac{1}2\lf\sqrt{4g+s^2}+\sqrt{4g+(2(g-1)-d+s)^2}\rf \\
& \leq  -g+1+\frac{d}{2}+\frac{1}2\lf\sqrt{4g+\left(\frac{g}2-2\right)^2}+\sqrt{4g+\left(2(g-1)-d+d-\frac{3g}2\right)^2}\rf \\ 
& = -\frac{g}2+3+\frac{d}2.
\end{align*}
Note that all `$=$' hold only when $d=2(g-1)$ and $s=d-\frac{3g}2$. But when $g$ is odd, we have $s\leq d-\frac{3g+1}2$, and therefore $h^0(C,E)\leq -\frac{g+1}2+3+\frac{d}2$. Hence, $h^0(C,E)\leq -\lf\frac{g+1}2\rf+3+\frac{d}2$.
\[\clf(E)\geq \frac{d}2+\lf\frac{g+1}2\rf-3-\frac{d}2+2=g-1-\lf\frac{g}2\rf.\]
Now assume there is no wall $\mathcal{W}$ for $\iota_*E$ and it is $\sigma_{0,\alpha}$-semistable where $\alpha \rightarrow 0$. Denote $x= d-2(g-1)$ and $p_1 = \overline{Z}(\iota_*E)$, so $\lVert op_1 \rVert = \sqrt{x^2+ 16g}$. Proposition \ref{prop:bound for global sections} implies that 
$$\clf(E) \geq g+1 - \dfrac{1}{2}\left\lfloor \sqrt{x^2+ 16g} \right\rfloor. $$
Thus for $-g+4 < x \leq 0$, we have $\clf(E) \geq g-1-\lf\frac{g}2\rf$. If $x \leq -g+4$, then again Proposition \ref{prop:bound for global sections} gives
$$2 h^0(C, E) \leq x +\sqrt{x^2+ 16g} \leq \dfrac{8g}{\sqrt{x^2+16g}} \leq 8. $$
Therefore the second part of Theorem \ref{main1} for $r=2$ follows by the fact that $\clf_2(C) \leq \clf_1(C) = g-1-\lf\frac{g}2\rf$.
\end{proof}
\subsection{Higher Picard number case}\label{sec:6.1}
Theorem \ref{main1} still holds when the ample divisor $H$ satisfies Assumption \star. 
\begin{description}
\item[Assumption (*)] $H^2$ divides $H.D$ for all curve classes $D$ on $X$.
\end{description}
We explain how to adapt all our arguments from  Picard rank one to this more general case.

 Let $\Lambda_H \cong \Z^3$ denote the image of the map
\[ \vv_H \colon K(X) \to \R^3, \quad E \mapsto \left(\rk(E), H\ch_1(E), \ch_2(E) \right).\]
Consider stability conditions for which the central charge factors via $\vv_H$, and denote the space of such stability conditions by $\Stab_H(X)$. The pair $\sab := \left(\Coh^\beta X, \zab\right)$ defines a stability condition on $D^b(X)$  and there is a continuous map from $\Gamma_+ \to \Stab_H(X)$. 
The slope function $\nab$ is defined in the same way. All the propositions in Section \ref{sect:background} hold for the higher Picard rank case. The Chern characters in part (\emph a) in Lemma \ref{lem:bound for beta1} should be modified to $H.\ch(F_1)=H^2$. All the other statements do not rely on the Picard rank.


\section{Smooth plane curves}
\label{section:p2}

Our method to control the dimension of global sections of semistable vector bundles (first part of Theorem \ref{main1}) can be generalized to curves on more general surfaces, especially for Fano surfaces. As a case study, we follow the argument for curves on K3 surfaces to set up a bound for smooth projective plane curves and finally compute their Clifford indices. We first review Bridgeland stability conditions on the projective plane.
\subsection{Review: space of geometric stability conditions on $D^b(\pp)$}
The space of geometric stability conditions on the projective plane $\pp$ is similar but slightly different with that of a K3 surface with Picard number one. In the projective plane case, the curve $\Gamma$ is replaced by the Le Potier curve (see \cite{Drezet-LePotier:P2,Izzet-Jack-Matthew,Li-spaceofp2,Chunyi-Xiaolei:birational}). Since the definition of Le Potier curve is rather involved, we will only use a simpler version $\tilde{\Gamma}$ which is enough for our purpose.

\begin{Def} \label{def:gcurvep2}
Let $\tilde{\gamma}:\R\rightarrow \R$ be a $1$-periodic function. When $x\in [-\frac{1}{2},\frac{1}{2}]$, 
\[ \tilde{\gamma}(x) := \begin{cases}
\frac{1}{2}x^2 - \frac{3}2 |x| +1 & \text{if $x\neq 0$} \\
0 & \text{if $x = 0$}.
\end{cases}\]
Let $\tilde{\Gamma}(x):=\frac{1}{2}x^2-\tilde{\gamma}(x)$. By abuse of notations, we also denote the graph of $\tilde{\Gamma}$ by the curve $\tilde{\Gamma}$.
\end{Def}
For $\beta \in \R$ and $\alpha > \tilde{\Gamma} (\beta)$, we define the central charge $\zab \colon K(\pp) \to \C$ as
\begin{equation} \label{eqn:Zab}
\zab(E) := -\ch_2(E) + \alpha \rk(E)  + i (\ch_1(E).H-\beta\rk(E)).
\end{equation}
By \cite[Proposition 1.10]{Li-spaceofp2}, we get a slice of stability conditions $\sab=(\Coh^\beta\pp,\zab)$ parametrized by $\tilde{\Gamma}_+$.  Results of stability condition and wall-crossings (Theorem \ref{thm:stabconstr}, Remark \ref{rem:stob} and Proposition \ref{prop:walls}) 
all hold without any change. One should be cautious that the end points of the first wall may not be on the curve $\tilde\Gamma$. 
\subsection{Upper bound on the dimension of global sections}
Let $C$ be a degree $l$ smooth irreducible curve in the projective plane $\pp$. Denote $\iota:C\hookrightarrow \pp$ the embedding morphism and $H \coloneqq \mathcal{O}_{\pp}(1)$. We recollect lemmas from the case of K3 surfaces. The next lemma generalizes \cite[Lemma 3.2]{Fe17} to objects in $D^b(\pp)$.
\begin{Lem}
\label{lem:p2length}
Fix an object $F\in \Coh^0\pp$ which is $\sigma_{0,\alpha}$-semistable for any positive real number $\alpha \ll 1$ and $\ch_1(F) \neq 0$. Then 
\begin{align*}
\hom(\mathcal O_{\pp}, F) \begin{cases}
=\rk(F) +\frac{3}2H\ch_1(F) +\ch_2(F) & \text{ when }\frac{\ch_2(F)}{H.\ch_1(F)} > -\frac{3}2, \\
\leq \rk(F) - \frac{\ch_1(F)^2}{2\ch_2(F)} & \text{ when } \ch_2(F)<0. 
\end{cases} 
\end{align*} 
\end{Lem}
\begin{proof}
We first assume $\frac{\ch_2(F)}{H\ch_1(F)}> -\frac{3}2$. The object $\cO_{\pp}(-3)[1]\in \Coh^0\pp$ is $\sigma_{0,\alpha}$-semistable and $\nu_{0,\alpha}(\cO_{\pp}(-3)[1]) = -\frac{3}2< \nu_{0,\alpha}(F)$, thus 
	$\Hom (F,\cO_{\pp}(-3)[i]) = 0,$
	for $i\leq 1$. By Serre duality, we have $\Hom(\cO_{\pp},F[2-i])=0$ for $i\leq 1$. Since both $F$ and $\cO_{\pp}$ are in the heart $\Coh^0\pp$, we have 
	$\Hom (\cO_{\pp},F[i]) = 0$,
	for $i\leq -1$.  Therefore, 
	\[\Hom (\cO_{\pp},F)=\chi(\cO_{\pp},F)=\rk(F)+\frac{3}2H.\ch_1(F)+\ch_2(F).\]
Now assume $\ch_2(F) <0$. Define the object $K \in D^b(\pp)$ as the canonical extension 
\[0\rightarrow F\rightarrow K\rightarrow \cO_{\pp}[1]\otimes \Ext^1(\cO_{\pp}[1],F)\rightarrow 0\]
in $\Coh^{0+}\pp$. We have $\ch(K)=(\rk(F)-h,\ch_1(F),\ch_2(F))$, where $h$ denotes $\dim\Ext^1(\cO_{\pp}[1],F)=\hom (\cO_{\pp},F)$. The object $K$ is semistable on the wall that the objects $F$ and $\mathcal{O}_X$ have the same phase, in particular, $\Delta(K)\geq 0$:
\begin{align*}
0\leq & (H.\ch_1(K))^2-2\ch_2(K)(\rk(F)-h)\\
\implies h \leq & \rk(F)-\frac{(H.\ch_1(K))^2}{2\ch_2(K)}=\rk(F)-\frac{(H.\ch_1(F))^2}{2\ch_2(F)}.
\end{align*}
\end{proof}

It is directly from the lemma that there is no such stable object with $\frac{\ch_2(F)}{H.\ch_1(F)}\in (-\frac{3}2,-1)$.

Define the function $L \colon (a,b) \in \mathbb{H} = \mathbb{R}  \times \mathbb{R}^{>0} \rightarrow \mathbb{R}^{>0}$ such that 
\begin{equation*}
L(a,b) = 
\begin{cases}
\dfrac{3}{2}b+a, & \text{ if } \;\;\; \dfrac{a}{b} \in I \coloneqq [-1, +\infty); \\
 -\dfrac{b^2}{2a}, & \text{ if } \;\;\; \dfrac{a}{b} \in J \coloneqq (-\infty , -1]. 
\end{cases} 
\end{equation*}
Note that $L(a,b) >0$ for any pair $(a,b) \in \mathbb{H}$. 
\begin{Lem} \label{norm}
	The function $L$ satisfies the triangle inequality in $\mathbb{H}$, in other words, for any two vectors $v_1 = (a_1,b_1)$ and $v_2 = (a_2,b_2)$ in $\mathbb{H}$, we have $L(v_1 +v_2) \leq L(v_1) +L(v_2)$. Moreover, L(kv) =k.L(v) for any $v \in \mathbb{H}$ and $k>0$.   
\end{Lem}
\begin{proof}
    The second claim follows clearly by definition. To prove the first claim, we consider four different  cases. 
    \begin{enumerate}
    \item  If both $a_1/b_1$ and $a_2/b_2$ are in $I = [-1, +\infty)$, then $L(v_1+v_2) = L(v_1)+L(v_2)$.
    \item If both $a_1/b_1$ and $a_2/b_2$ are in $J = (-\infty , -1]$, then  
	\begin{equation*}
	0 \leq \dfrac{-1}{a_1+a_2}\bigg(b_1^2\bigg(\dfrac{a_2}{a_1}\bigg) + b_2^2\bigg(\dfrac{a_1}{a_2} \bigg) -2b_1b_2\bigg).
	\end{equation*}
	This implies 
	\begin{equation*}
	0 \leq  b_1^2 \bigg(\dfrac{1}{a_1+a_2} - \dfrac{1}{a_1} \bigg) + b_2^2 \bigg(\dfrac{1}{a_1+a_2} - \dfrac{1}{a_2} \bigg) + \dfrac{2b_1b_2}{a_1+a_2} = -L(v_1+v_2) +L(v_1) +L(v_2). 
	\end{equation*}
    \item If $a_1/b_1 \in I$, $a_2/b_2 \in J$ and $(a_1+a_2)/(b_1+b_2) \in I$, then since $a_2/b_2 \leq -1$, we have 
	\begin{equation*}
	\dfrac{3}{2}b_2+a_2 \leq -\dfrac{b_2^2}{2a_2} 
	\end{equation*} 
	which implies 
	\begin{equation*}
	L(v_1 +v_2) = \dfrac{3}{2}(b_1+b_2) + a_1+a_2 \leq   \dfrac{3}{2}b_1+a_1 + -\dfrac{b_2^2}{2a_2} = L(v_1) + L(v_2)
 	\end{equation*}
    \item  If $a_1/b_1 \in I$, $a_2/b_2 \in J$ and $(a_1+a_2)/(b_1+b_2) \in J$, then there is a non-negative  real number $k<1$ such that $(a_1+ka_2)/(b_1+kb_2) = -1$, then case (c) implies that 
	\begin{equation*}
	L(v_1 +kv_2) \leq L(v_1) + k.L(v_2).
	\end{equation*}	
	Therefore, case (b) gives
	\begin{equation*}
	L(v_1+v_2) \leq L\big((1-k)v_2\big) + L(v_1 +kv_2) \leq (1-k).L(v) + k.L(v) + L(v_2),
	\end{equation*}
    \end{enumerate}
	which proves the claim.
\end{proof}
Fix a semistable rank $r$-vector bundle $E$ of degree $d$ on the curve $C$. The same argument as in \cite[Proposition 3.4]{Fe17} implies that there exists $\epsilon >0$ such that the Harder-Narasimhan filtration of $\iota_*E$ is a fixed sequence 
\begin{equation*}
0 = \tilde{E}_0 \subset \tilde{E}_1 \subset ... \subset \tilde{E}_{n-1} \subset \tilde{E}_n = \iota_*E
\end{equation*}
for all stability conditions $\sigma_{0,\alpha}$ where $0 < \alpha < \epsilon$. Let $P_{\iota_*E}$ be the polygon with the extremal points $p_i \coloneqq \big( \ch_2(\tilde{E}_i) , \ch_1(\tilde{E}_i) \big) \in \mathbb{R}^2$ for $i=0,...,n$. Then Lemma \ref{lem:p2length} implies that
\begin{equation}\label{upper bounds in pp}
h^0(X,\iota_*E) \leq \rk(E) + \sum_{i=1}^{n} L(\overrightarrow{p_ip_{i-1}}).
\end{equation}
Note that by definition, the curve with the equation $y=x^2/2$ is above the curve $\tilde{\Gamma}$. Also when $0 \leq x <1$, the function $\tilde{\Gamma}(x)\leq -\frac{1}2x$. Therefore, any point $(\beta,\alpha)$ in the gray area in Figure \ref{fig:p2wall} gives a Bridgeland stability condition $\sigma_{\beta,\alpha}$. 

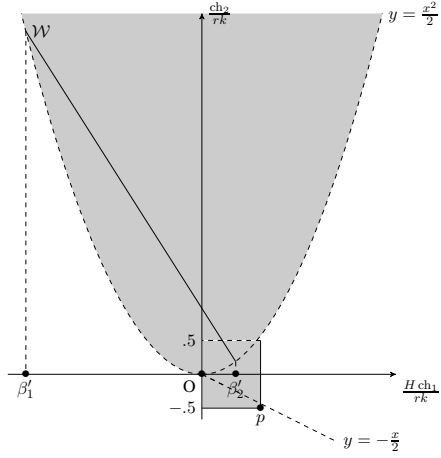
\begin{figure}[h]
\begin{center}
\tikzset{%
    add/.style args={#1 and #2}{
        to path={%
 ($(\tikztostart)!-#1!(\tikztotarget)$)--($(\tikztotarget)!-#2!(\tikztostart)$)%
  \tikztonodes},add/.default={.2 and .2}}
}
\scalebox{0.6}{
\begin{tikzpicture}
     \newcommand\bt{0.75}
     \newcommand\bo{-3.9}
     
\fill [fill=gray!40!white] (0,0) parabola (4,8) parabola [bend at end] (-4,8) parabola [bend at end] (0,0); 
  \filldraw[fill=gray!40!white, draw=white] (0,0)--(0,-.75)--(1.3,-.75)--(1.3,.75)--(0,.75)--(0,0); 
\draw [dashed] (0,0) parabola (4,8); 
			\draw [dashed] (0,0) parabola (-4,8); 
			
       \coordinate (B2) at (\bt,\bt*\bt/2);
       
       \draw[dashed] (B2) -- (\bt,0) node {$\bullet$} node[below] {$\beta_2'$};
       \coordinate (B1) at (\bo,\bo*\bo/2);
       
       \draw[dashed] (B1) -- (\bo,0) node {$\bullet$} node[below] {$\beta_1'$};

\draw[->] [opacity=1] (-4.3,0) -- (0,0) node[above right] {}-- (4.3,0) node[below right, opacity =1] {$\frac{H\ch_1}{rk}$};



\draw[->][opacity=1] (0,-1)-- (0,0) node {$\bullet$} node [below left] {O} --  (0,8) node[right, opacity=1] {$\frac{\ch_2}{rk}$};



    \draw [add = 0 and 0] (B2) to (B1) node[above,right] {$\mathcal{W}$};
    \draw (4,8) node[right] {$y=\frac{x^2}{2}$};
    \draw[dashed] (0,0) -- (0.5,-0.25);
    \draw[dashed] (0,0)--(3,-1.5) node[right]{$y=-\frac{x}2$};
    \draw (1.3,-.75)--(1.3,.75);
    \draw[dashed] (1.3,.75)--(0,.75) node[left] {$.5$};
    \draw (1.3,-.75)--(0,-.75) node[left] {$-.5$};;
    \draw (1.3,-.75) node {$\bullet$} node[below] {$p$};
    

\end{tikzpicture}
}
\end{center}
\caption{First wall for $\iota_*E$.} \label{fig:p2wall}
\end{figure}

\begin{Lem}[Lemma \ref{lem:ph}]
\label{lem:p2phasebound} 
For any semistable factor $E_i \coloneqq \tilde{E}_i/\tilde{E}_{i-1}$ in the Harder-Narasimhan filtration of $\iota_*E$, we have the slope $\frac{\ch_2(E_i)}{H\ch_1(E_i)}\in \left[\frac{d}{2rl}-\frac{l}2,\frac{d}{2rl}\right]$. When $d<rl$, the slope is either in the range $\left[\frac{d}{2rl}-\frac{l}2,-\frac{1}2\right]$ or $\left[-\frac{l-1}2,\frac{d}r-l+\frac{1}2\right]$.
\end{Lem}
\begin{proof}
Let $0 \rightarrow F_2 \rightarrow \iota_*E \rightarrow F_1 \rightarrow 0$ be the destabilizing sequence at the wall $\cW$ for $\iota_*E$ which passes a stability condition of form $\sigma_{0,\alpha}$. We have $\ch_1(\iota_* E) = rlH$ and $\ch_1(H^0(F_1)) = alH$ for some integer $a \geq 0$. Denote $\rk(H^{-1}(F_1)) = \rk(F_2) = s$, $\ch_1(H^{-1}(F_1)) = d_1H$ and $\ch_1(F_2) = d_2H$. Let $T(F_2)$ be the maximal torsion subsheaf of $F_2$, then $\ch_1(T(F_2)) = tlH$ for some integer $t \geq 0$. The same argument as in the first part of Lemma \ref{lem:fw} implies that  
\begin{equation*}
rl-al \leq sl+tl.
\end{equation*}
Therefore, 
\begin{equation}\label{in.diff}
\dfrac{\ch_1\big(F_2/T(F_2)\big).H}{sH^2} - \dfrac{\ch_1\big(H^{-1}(F_1)\big).H}{sH^2} = \dfrac{d_2-tl -d_1}{s} = \dfrac{rl-al-tl}{s} \leq l.
\end{equation}
Now assume the wall $\mathcal{W}$ intersects the parabola with the equation $y=x^2/2$ at two points $(\beta_2' , \beta_2'^2/2)$ and $(\beta_1', \beta_1'^2/2)$ where $\beta_1' <0 < \beta_2'$. By applying the same argument as in Lemma \ref{lem:fw}, the inequality \eqref{in.diff} gives $\beta_2'-\beta_1'\leq l$. Proposition \ref{prop:walls} implies that the slope of the wall $\mathcal{W}$ is 
\begin{equation*}
\dfrac{\frac{1}2(\beta_2')^2 - \frac{1}2(\beta_1')^2}{\beta_2'-\beta_1'} = \dfrac{\beta_2'+\beta_1'}{2} = \dfrac{\ch_2(\iota_*E)}{H\ch_1(\iota_*E)} = \dfrac{d-r\frac{l^2}2}{rl} = \dfrac{d}{rl} - \dfrac{l}{2}.
\end{equation*}
Therefore, $\beta_2' \leq \dfrac{d}{rl}$ and $\beta_1' \geq \dfrac{d}{rl} -l$. By a similar argument as in Lemma \ref{lem:ph}, one can show that for each of the Harder-Narasimhan factors $E_i$, 
\begin{equation*}
\dfrac{\beta_1'}{2} = \dfrac{\beta_1'^2/2}{\beta_1'} \leq 
\frac{\ch_2(E_i)}{H\ch_1(E_i)} \leq \dfrac{\beta_2'^2/2}{\beta_2'} = \dfrac{\beta_2'}{2}. 
\end{equation*}
Thus the first claim follows. \\
Now assume $d <rl$, so $\beta_2' <1$. If the wall $\mathcal{W}$ intersects the line with the equation $x=-1$ at a point $(1,y)$ for $-1/2 < y < 1/2$, then the same argument as in Lemma \ref{lem:fw} implies that 
$$1 \leq \dfrac{\ch_1\big(F_2/T(F_2)\big).H}{sH^2}$$
and inequality \eqref{in.diff} implies that 
$$1-l \leq   \dfrac{\ch_1\big(H^{-1}(F_1)\big).H}{sH^2} \leq \beta_1.$$ Therefore the wall $\mathcal{W}$ is below the line $L$ which has the same slope as $\mathcal{W}$ and passes through the point $(1-l,\frac{(l-1)^2}2)$. The line $L$ intersects the line $x=-1$ at the point $(1,\frac{d}r-l+\frac{1}2)$. Thus the same argument as that in Lemma \ref{lem:ph} shows that each slope $\frac{\ch_2(E_i)}{H\ch_1(E_i)}$ is in the range $[-\frac{1-l}2,\frac{d}r-l+\frac{1}2]$.\\
If we have $y<-1/2$, then the wall intersects the line segment $\overline{op}$ which has slope $-1/2$, see Figure \ref{fig:p2wall}. Thus the same argument as in Lemma \ref{lem:ph} implies that $\frac{\ch_2(E_i)}{H\ch_1(E_i)} \leq -\frac{1}2$ and the second claim follows. 
\end{proof}
\begin{Thm}\label{thm:p2}
	Let $C$ be a degree $l(\geq 5)$ smooth irreducible curve on the projective plane. Let $E$ be a semistable vector bundle with rank $r$ and degree $d$ such that $0\leq d\leq rl(l-3)/2$. Then
	\[\dim H^0(C,E) \leq \begin{cases}
	r+\left(\frac{3}{2l} + \frac{d}{2rl^2}\right)d& \text{if $d \geq rl$} \\
	\max\{3r+d-rl,r+\frac{rl+r}{rl^2-d}d\} & \text{if $r(l-1)\leq d <rl$}
	\end{cases}\]
\end{Thm}
\begin{proof}
When $d\geq rl$, Lemma \ref{lem:p2phasebound} implies that the polygon $P_{\iota*E}$ is inside the triangle $o\tilde{p}q$ where $ \tilde{p} = (\frac{d^2}{2rl^2}, \frac{d}l)$ and $q = (-\frac{rl^2}2 +d, rl)$. Then Lemma \ref{norm} and convexity of the polygon $P_{\iota*E}$ imply that  
\begin{align*}
h^0(C,E)& =\hom(\cO_{\pp},\iota_* E) \leq \sum_{i=1}^{n}L(\overrightarrow{p_ip_{i-1}})  \\ 
& \leq L(\overrightarrow{o\tilde{p}}) + L(\overrightarrow{\tilde{p}q}) = \frac{3d}{2l} +\frac{d^2}{2rl^2}+\frac{(rl-\frac{d}l)^2}{rl^2-2d+\frac{d^2}{rl^2}}\\
& = \frac{3d}{2l} +\frac{d^2}{2rl^2} + r.
\end{align*} 
When $d<rl$,  if the range of the slopes in Lemma \ref{lem:p2phasebound}  is given by $\left[\frac{d}{2rl}-\frac{l}2,-\frac{1}2\right]$, then we may let $\tilde{p}$ be at $\left(\frac{-rdl}{2rl^2-2rl-2d},\frac{rdl}{rl^2-rl-d}\right)$. Therefore,  
\begin{align*}
h^0(C,E)& \leq L(\overrightarrow{o\tilde{p}}) + L(\overrightarrow{\tilde{p}q})= \frac{rdl}{rl^2-rl-d}  +\frac{\left(rl-\frac{rdl}{rl^2-rl-d}\right)^2}{rl^2-2d-\frac{rdl}{rl^2-rl-d}}\\
& = \frac{rdl}{rl^2-rl-d}  +\frac{rl}{rl^2-d}\left(rl-\frac{rdl}{rl^2-rl-d}\right)\\
& = r+d\frac{rl+r}{rl^2-d}.
\end{align*} 
Also if the range of the slopes in Lemma \ref{lem:p2phasebound}  is given by $\left[-\frac{l-1}2,\frac{d}r-l+\frac{1}2\right]$, then we may let $\tilde{p}$ be at $\left(d-rl+\frac{r}2,r\right)$. Therefore,  
\begin{align*}
h^0(C,E)& \leq L(\overrightarrow{o\tilde{p}}) + L(\overrightarrow{\tilde{p}q})= \frac{3}2r+d-rl+\frac{r}2  +\frac{r^2(l-1)^2}{rl^2-2rl+r}=3r+d-rl,
\end{align*} 
which completes the proof.
\end{proof}
As an interesting consequence, part i) of Mercat conjecture (\cite{Mercat2002CLIFFORD}) holds for smooth plane curves.
\begin{Cor}\label{cor:mercatforp2}
	Let $C$ be a degree $l(\geq 5)$ smooth irreducible plane curve, then 
	\[\clf_r(C)=l-4,\]
	for any $r$.  
\end{Cor}
\begin{proof}
	Let $E$ be a semistable vector bundle with rank $r$ and degree $l$, when $d\geq rl$, by Theorem \ref{thm:p2}
	\begin{align*}
	\clf(E) & \geq \frac{d}r-\frac{2}r\left(\frac{3}{2l} + \frac{d}{2rl^2}\right)d  \geq \min_{d=rl,d=rl(l-3)/2}\left\{\frac{d}r-\frac{2}r\left(\frac{3}{2l} + \frac{d}{2rl^2}\right)d\right\} \\ 
	& = \min\left\{l-2l\left(\frac{3}{2l} + \frac{rl}{2rl^2}\right),\frac{l(l-3)}2-\left(\frac{3}{2l} + \frac{l-3}{4l}\right)l(l-3)\right\} \\ 
	& = \min\left\{l-4,\frac{l^2-6l+9}4\right\}=l-4.
	\end{align*}
    When $\frac{l^2-l}{l+1}r\leq d< rl$ and the upper bound for $H^0(C,E)$ is given by $3r+d-rl$  in Theorem \ref{thm:p2}, then
	\begin{align*}
	\clf_r(E) & \geq \frac{d}r- \frac{2}r\left(2r+d-rl\right)=l-4+l-\frac{d}r>l-4.
	\end{align*}
	When $\frac{l^2-l}{l+1}r\leq d< rl$ and the upper bound for $H^0(C,E)$ is given  by $r+\frac{rl+r}{rl^2-d}d$ in Theorem \ref{thm:p2}, then
	\begin{align*}
	\clf_r(E) & \geq \frac{d}r- \frac{2(l+1)}{r(l^2-l)}d=\frac{d}r(1-\frac{2l+2}{l^2-l})\geq \frac{l^2-l}{l+1}(1-\frac{2l+2}{l^2-l})>l-2-2.
	\end{align*}
	When $d<\frac{l^2-l}{l+1}r$, by Theorem \ref{thm:p2}, 
	$\dim H^0(C,E)<r+\frac{rl+r}{rl^2-rl}d<2r$.
	On the other side, one may take $E=\cO_C(1)^{\oplus r}$, then $\clf(E)=l-4$. Therefore, we have $\clf_r(C)=l-4$.
\end{proof}

\bibliography{all}                      
\bibliographystyle{halpha}     

\end{document}